\documentclass[a4paper]{amsart}

\usepackage[utf8]{inputenc}
\usepackage[T1]{fontenc}
\usepackage[normalem]{ulem}
\usepackage{cancel}

\usepackage[pdftex]{graphicx}
\usepackage{latexsym}
\usepackage{amsmath,amssymb,amsthm}
\usepackage{mathtools}
\usepackage{etoolbox}
\usepackage[shortlabels]{enumitem}   %
\usepackage{xcolor}
\usepackage{mathpazo} %
\usepackage{nicefrac} %
\allowdisplaybreaks %

\usepackage{scalerel,stackengine} 
\stackMath
\newcommand\reallywidehat[1]{%
\savestack{\tmpbox}{\stretchto{%
  \scaleto{%
    \scalerel*[\widthof{\ensuremath{#1}}]{\kern-.6pt\bigwedge\kern-.6pt}%
    {\rule[-\textheight/2]{1ex}{\textheight}}%
  }{\textheight}%
}{0.5ex}}%
\stackon[1pt]{#1}{\tmpbox}%
}

\usepackage{standalone}

\usepackage[colorlinks=true,urlcolor=blue]{hyperref} 			%
\hypersetup{
	colorlinks,
	linkcolor={red!50!black},
	citecolor={blue!50!black},
	urlcolor={blue!80!black}
}
\usepackage[msc-links]{amsrefs}

\usepackage[noabbrev,capitalise]{cleveref}

\newtheorem{Theorem}{Theorem}[section]
\newtheorem{Lemma}[Theorem]{Lemma}		 
\newtheorem{Corollary}[Theorem]{Corollary}	
\newtheorem{Assumption}[Theorem]{Assumption} 
\theoremstyle{definition}
\newtheorem{Remark}[Theorem]{Remark}

\newcommand{\R}{\mathbb{R}} %
\newcommand{\N}{\mathbb{N}} %

\newcommand{\half}{1/2}
\newcommand{\mfrac}[2]{\scalebox{0.8}{$\dfrac{#1}{#2}$}{}}

\newcommand{\grad}{\nabla}
\renewcommand{\div}{\nabla \cdot}
\newcommand{\laplace}{\Delta}
\newcommand{\dintx}{\, \mathrm{d} x}

\newcommand{\dints}{\, \mathrm{d} \sigma}
\newcommand{\dint}[1]{\, \mathrm{d} #1}
\newcommand{\supp}{\operatorname{supp}}
\newcommand{\interior}{\operatorname{int}}

\newcommand{\defpnt}{{\hspace*{0.15em}:\hspace*{0.15em}}}
\newcommand{\htar}{h_{\mathrm{tar}}}

\def\L#1#2{{L^{#1}({#2})}}

\newcommand{\abs}[1]{\left|#1\right|}

\newcommand{\norm}[1]{\left\lVert #1 \right\rVert}
\newcommand{\bignorm}[1]{\big\lVert #1 \big\rVert}

\DeclareFontFamily{U}{matha}{\hyphenchar\font45}
\DeclareFontShape{U}{matha}{m}{n}{
<-6> matha5 <6-7> matha6 <7-8> matha7
<8-9> matha8 <9-10> matha9
<10-12> matha10 <12-> matha12
}{}
\DeclareSymbolFont{matha}{U}{matha}{m}{n}

\DeclareFontFamily{U}{mathx}{\hyphenchar\font45}
\DeclareFontShape{U}{mathx}{m}{n}{
<-6> mathx5 <6-7> mathx6 <7-8> mathx7
<8-9> mathx8 <9-10> mathx9
<10-12> mathx10 <12-> mathx12
}{}
\DeclareSymbolFont{mathx}{U}{mathx}{m}{n}

\DeclareMathDelimiter{\vvvert} {0}{matha}{"7E}{mathx}{"17}%

\DeclarePairedDelimiterX{\energynormfacotry}[1]
{\vvvert}
{\vvvert}
{\ifblank{#1}{\:\cdot\:}{#1}}

\newcommand{\energynorm}[1]{\energynormfacotry*{#1}}
\newcommand{\bigenergynorm}[1]{\energynormfacotry[\big]{#1}}

\newcommand{\ip}[2][]{
	\ifthenelse{ \equal{#1}{} }
	{\bigl( #2 \bigr)}
	{\bigl( #2 \bigr)_{#1}}
}
\newcommand{\abilSymb}{a}
\newcommand{\abil}[2][]{
	\ifthenelse{ \equal{#1}{} }
	{\abilSymb\bigl( #2 \bigr)}
	{\abilSymb_{#1}\bigl( #2 \bigr)}
}

\newcommand\restr[2]{{          %
\left.\kern-\nulldelimiterspace %
#1                              %
\vphantom{\big|}                %
\right|_{#2}                    %
}}
\newcommand{\bigrestr}[2]{#1\big|_{#2}}

\DeclarePairedDelimiterX\setc[2]{\{}{\}}{\,#1 \;\delimsize\vert\; #2\,}
\DeclarePairedDelimiter\monosetc{\{}{\}}

\newcommand{\Th}[1][]{
	\ifthenelse{ \equal{#1}{} }
	{\mathcal{T}_h}
	{\mathcal{T}_h(#1)}
} 
\newcommand{\hmin}{h_{\min}}

\newcommand{\PkK}{\mathbb{P}_k(K)}
\newcommand{\PkTh}[1][]{\mathcal{P}_k(\Th[#1])}

\newcommand{\PkZeroTh}[1][]{\mathcal{P}_{k,0}(\Th[#1])}

\newcommand{\LocInt}{\mathcal{I}_K}
\newcommand{\LocIntRef}{\mathcal{I}_{\widehat{K}}}
\newcommand{\NodalInt}{\mathcal{I}_h}

\newcommand{\fulldomain}{\Omega}
\newcommand{\dfulldomain}{\partial \Omega}

\newcommand{\Dh}{D_h}
\newcommand{\Patch}[2]{P_{#1}(#2)}
\newcommand{\bLayer}[1]{\Sigma_{#1}}
\newcommand{\ellmax}{\ell_\text{max}}

\newcommand{\Cutoff}{ \eta_\ell}  %
\newcommand{\IntCutoff}{ \raisebox{0.2em}{$\chi$}_{h}^{\ell}} %

\newcommand{\CCutoff}{ c_\eta}   %

\newcommand{\CInth}{c_{\mathcal{I}}}  %

\ifpdf
\hypersetup{
  pdftitle={A discrete Saint-Venant principle for finite element discretizations of elliptic problems},
  pdfauthor={Tim Buchholz and Julian Dörner}
}
\fi

\begin{document}

\title[Discrete Saint-Venant principle for FEM discretizations]{A discrete Saint-Venant principle for finite element discretizations of elliptic problems}

\author{Tim Buchholz}
\author{Julian Dörner}
\address{T. Buchholz, J. Dörner \hfill\break 
Institute for Applied and Numerical Mathematics,\hfill\break
Karlsruhe Institute of Technology (KIT), \hfill\break
D-76128 Karlsruhe, Germany}
\email{\{tim.buchholz,julian.doerner\}@kit.edu}

\begin{abstract}
  The present paper studies finite element discretizations of second-order elliptic boundary value problems with homogeneous right-hand side and inhomogeneous boundary conditions. 
  We establish discrete spatial decay estimates on element patches for the energy norm of the discrete solution, showing that the influence of boundary data decays exponentially away from the boundary. 
  The resulting estimates are a discrete analog of Saint-Venant-type principles and provide a rigorous foundation for localization arguments in finite element methods. 
  As an application, we present how these results can be employed in the convergence analysis of domain decomposition methods, on the example of the discrete parallel Schwarz method.
  Finally, the findings are thoroughly demonstrated on several numerical examples.
\end{abstract}



\maketitle

\section{Introduction}
\label{sec:Introduction}

The spatial decay of boundary-induced effects in elliptic problems can be traced back to the Saint-Venant principle, originally formulated in the context of linear elasticity in the mid-19th century \cite{StV1856}. 
A classical and widely cited description states that \textit{'...the difference between the effects of two different but statically equivalent loads becomes very small at sufficiently large distances from the load'}, cf. \cite{Lov44}. 
Early rigorous investigations focused on linear elasticity and established decay properties of stresses and strains away from regions of load application, see, e.g., \cite{Tou65,Kno66}.

It was soon recognized that the underlying mechanism is not specific to elasticity but rather a consequence of ellipticity. 
This insight led to extensions of Saint-Venant-type results to general second-order elliptic equations, where spatial decay of solutions was established primarily in a pointwise sense and under strong regularity assumptions, cf. \cite{Kel74,WheHo76,HorWh77}. 
Among these developments, the work of Mieth \cite{Mieth75,Mieth76} is particularly relevant for the present paper, as it derives a decay estimate at the level of the energy norm. 
Since our analysis focuses on energy quantities, this perspective is the important analytical motivation of our work. Thus, we repeat the argument in brief below, see \Cref{subsec:analyticalMotivation}.
 
Comprehensive perspectives on the Saint-Venant principle and its extensions were later provided in a series of review articles by Horgan \cite{HorKn83,Hor89,Hor96}, emphasizing both the physical origins and the mathematical structure of the principle. 
While these classical results are formulated at the continuous level, they establish the conceptual foundation for spatial energy decay as a generic feature of elliptic operators.

In the present work, we adopt this viewpoint and investigate analogous decay phenomena directly at the discrete level. 
Our results may thus be interpreted as a finite element analogue of Saint-Venant-type principles, providing quantitative spatial energy decay estimates for discretizations of elliptic problems with homogeneous right-hand side and inhomogeneous boundary conditions.

Spatial decay properties play a central—though often implicit—role in the analysis of domain decomposition (DD) methods; see, e.g., \cite{TosWi05,Gan08,GanZh22}. 
While many modern DD approaches are formulated at the matrix level, their analysis is frequently guided by continuous decay arguments, as discussed in detail in \cite{Gan08}.
In particular, superlinear convergence results for Schwarz waveform relaxation methods rely on analytic decay properties of solutions to elliptic problems on unbounded or semi-infinite domains, cf. \cite{Gan97,GilKe02}.

Recent developments in DD methods for time-harmonic wave propagation have been made; see \cite{ClaPar2022}. The authors formulate the problem on the skeleton of the subdivided domain and define exchange operators that govern the communication between subdomains. Many possible choices of such exchange operators are studied in \cite[Sec.~5]{ClaPar2022}. A particularly interesting choice for our work is a global exchange operator, which exhibits excellent stability properties but appears at first glance to be incompatible with a fast algorithm. In \cite{ClaAtc25}, it is numerically demonstrated that a truncation of this operator is sufficient and enables fast evaluation.
This operator fits in the class of problems studied below.

More recent surveys, including work on time-parallel methods for parabolic and hyperbolic problems, further underline the importance of decay phenomena in both theoretical and practical contexts \cite{GanWZ25}. 
Related localization arguments also arise in tent-pitching methods combined with implicit time discretizations (cf.~\cite{GopSW17}), where the causality constraints limiting the admissible height of space--time tents may be interpreted as reflecting a localized influence of boundary and interface data.
Finally, similar energy-based decay statements appear in \cite{GalMa23} and constitute one of the main analytical tools in our own recent work \cite{BucH25}.

Decay estimates are most explicitly exploited in numerical homogenization and multiscale methods. 
In the context of the localized orthogonal decomposition (LOD), Malqvist and Peterseim \cite{MalP14} proved the essential decay of corrector functions via Caccioppoli-type inequalities, justifying the truncation of global correctors to local patches; see also \cite{MalP21}. 
Closely related principles appear in generalized finite element methods (GFEM) and multiscale GFEM variants \cite{MaScDo22,AlbMaSc25}, while alternative decay arguments based on Green's functions have been explored in \cite{Glo11}.
Decay estimates of a similar nature also underlie localization phenomena for eigenstates of Schr\"odinger operators, which are closely related to the exponential decay of the associated Green's functions, as shown in \cite{AltHP20}.

These works demonstrate that spatial energy decay is a fundamental structural property underlying many modern numerical methods.
The present paper aims to strengthen this perspective by establishing such decay directly for finite element discretizations, following the arguments of the aforementioned analytical results. 
We prove a fully discrete Saint-Venant-type principle with explicit mesh- and problem-dependent constants and demonstrate that its asymptotic behavior agrees with the decay rate of the analytical solution.
We demonstrate the results on an overlapping domain decomposition algorithm.
Furthermore, we thoroughly validate our results numerically, verify the dependence of problem-dependent parameters on the decay rate, and demonstrate their asymptotic optimality.

\subsection*{Outline}
The paper is organized as follows. 
We close this section below with the analytical motivation on a simple example.
\Cref{sec:Preliminaries} introduces the model problem and its variational formulation, the conforming finite element discretization, and the definition of local energies on element patches.
Following that, we state in \Cref{sec:MainResults} our main results, i.e., the discrete spatial decay estimate for the finite element approximation, and show that its asymptotic behavior agrees with that of the corresponding continuous solution.
In \Cref{sec:Cutoff}, we build the necessary tools for the proofs of the main results in \Cref{sec:ProofMainResults}.
\Cref{sec:Applications} demonstrates how the decay estimate can be
used in the convergence analysis of an overlapping domain decomposition
method.
Finally, \Cref{sec:NumericalValidation} presents numerical experiments
that investigate the parameter dependence and asymptotic behavior
predicted by the theory.

\subsection{Analytical motivation} 
\label{subsec:analyticalMotivation}

We briefly illustrate the proof technique on the continuous level to motivate our  subsequent work.
For simplicity, we choose a simplified setting which lags no descriptiveness and refer to the literature mentioned above for a more general setting; especially to the work from Mieth and Horgan, in particular \cite[Chapter 2.3.1]{Mieth75} and \cite[Section II.D]{HorKn83}.

Let \(\Omega\subset \R^2\) denote a smooth domain, star shaped with respect to the origin.
The function \(u\in C^2(\overline{\Omega})\) denotes the strong solution of the problem
\begin{equation}
    \label{eq:motivationProblem}
    -\laplace u +\lambda^2 u = 0
    \,\, \text{ in } \Omega,
    \quad u = g\,\, \text{ on } \Gamma = \partial\Omega,
\end{equation}
where \(g\) is a sufficiently smooth boundary datum.
Centered around the origin, we define a series of subdomains of \(\Omega\) as
\begin{equation*}
    \Omega(s) = \setc{sx}{x\in\Omega},\quad 0 < s \leq 1
    \qquad
    \Omega(1) = \Omega,
\end{equation*}
cf. \Cref{fig:SubDomainSeries}.
We define the energy functional of the solution as
\(
    E(s) = \int_{\Omega(s)} |\nabla u|^2 + \lambda^2 u^2 \dintx
\)
and prove the result
\begin{equation}
    \label{eq:DesiredResult}
     E(s) \leq e^{-2\lambda(1 - s)}E(1), \,\,\text{ for } 0 < s \leq 1.
\end{equation}
Hence, a smaller sub-domain has an exponentially smaller energy concentrated on it.

\begin{figure}
    \centering
    \includegraphics{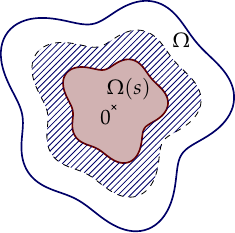}
    \caption{A smooth domain $\Omega$, star shaped with respect to the origin. A subdomain $\Omega(s)$, for \(s>0\) is drawn in red, while the hatched blue area depicts $\Omega(s+\delta)\setminus\Omega(s)$.}
    \label{fig:SubDomainSeries}
\end{figure} 

The proof follows two main steps, which we later replicate in the discrete setting.
We first derive a differential inequality for the energy.
By integration by parts, eq.~\eqref{eq:motivationProblem}, and Young's inequality, we obtain
\begin{equation*}
    E(s) 
        = \int_{\Gamma(s)}u \nabla u \cdot n \dints
        \leq \frac{1}{2\lambda} \int_{\Gamma(s)} |\nabla u|^2 + \lambda^2 u \dints.
\end{equation*}
Integration of this inequality with \(\delta = \frac{1}{2}(1 - s)\) shows that
\begin{equation*}
    \int_{s}^{s+\delta} E(t)\dint{t}
        \leq \frac{1}{2\lambda} \int_{\Omega(s+\delta)\smallsetminus \Omega(s)}|\nabla u|^2 + \lambda^2 u \dints
        = \frac{1}{2\lambda}(E(s+\delta) - E(s)). 
\end{equation*}
Thus, we obtain the differential inequality
\begin{equation*}
    -\frac{\mathrm{d}}{\mathrm{d}s} \int_{s}^{s+\delta} E(t) \dint{t} 
        + 2\lambda \int_{s}^{s+\delta} E(t) \dint{t}
        \leq 0.
\end{equation*}
In the next step, we solve the differential inequality and use monotonicity of the energy functional to obtain
\begin{equation*}
    \delta E(s) 
        \leq \int_{s}^{s+\delta} E(t) \dint{t} 
        \leq e^{-2\lambda(1 -s)}\int_{s+\delta}^{1}E(t)\dint{t}
        \leq \delta e^{-2\lambda(1 -s)}E(1).
\end{equation*}
Division by \(\delta\) then provides the desired result \eqref{eq:DesiredResult}.

\bigskip

\section{Preliminaries}
\label{sec:Preliminaries}
\subsection{Sobolev spaces \& traces}
In this section, we  introduce the functional analytic framework.
Let \(\Omega\subset \R^d\) be a Lipschitz domain.
Throughout the  manuscript, we make use of the Lebesgue space \(L^2(\Omega)\) of square integrable functions and denote its inner product and norm with \(\bigl(\cdot,\cdot\bigr)_{\Omega}\) and \(\norm{\cdot}_{L^2(\Omega)}\), respectively.

We denote by \(H^1(\Omega)\) the Sobolev space of functions with weak derivatives up to order one in \(L^2(\Omega)\).
The closure of the compactly supported test functions \(C^\infty_c(\Omega)\) in \(H^1(\Omega)\) is denoted by \(H^1_0(\Omega)\).
For the usual norm and semi-norm of \(H^1(\Omega)\) we write \(\norm{\cdot}_{H^1(\Omega)}\) and \(\abs{\cdot}_{H^1(\Omega)}\), respectively.

It is well-known that every function \(u\in H^1(\Omega)\) admits a well-defined trace \(\restr{u}{\partial \Omega} \in H^{\half}(\partial \Omega)\), where the latter space denotes the fractional Sobolev space of order one half on the boundary \(\partial \Omega\);
see, for example, \cite[Def.~2.14,Thm.~3.10]{ErnG21I}.
In the model problem considered below, we will make use of functions that are only defined on a relatively open Lipschitz part of the boundary, i.e., \(\Gamma \subset \partial \Omega\).
Thus, we define the Lions-Magenes space as 
\[
    H^{\half}_{00}(\Gamma) = \setc{g\in H^{1/2}(\Gamma)}{\widetilde{g} \in H^{1/2}(\partial \Omega)},
    \qquad \norm{g}_{H^{\half}_{00}(\Gamma)} = \norm{\widetilde{g}}_{H^{\half}(\partial \Omega)},
\]
where \(\widetilde{g}\) denotes the extension of $g$ by zero to the whole of \(\partial \Omega\).
This space is alternatively defined via interpolation of Hilbert spaces, and we refer to \cite{LionsMagenes72} for a thorough treatment.
Infamous for their abstractness, we emphasize that \(H^{\half +\epsilon}_{0}(\Gamma) \subset H^{\half}_{00}(\Gamma)\) for any \(\epsilon > 0\).

\subsection{Problem \& variational formulation}\label{subsec:ProblemAndVariational}
We now introduce the model problem.
Let \(\fulldomain \subset \mathbb{R}^d\) be a bounded and simply connected polyhedral domain.
We denote with \(\Gamma\subseteq \dfulldomain\) a simply connected subset of the boundary and define the complement \(\Gamma^C = \dfulldomain \setminus \Gamma\).
Let $\lambda \geq 0$ and $A \in \L{\infty}{\fulldomain, \R^{d\times d}}$ be a positive and bounded coefficient that satisfies $\alpha \leq A(x) \leq \beta$ almost everywhere in \(\fulldomain\) for positive constants $\alpha,\beta > 0$.

For a given  \(g\in H^{\half}_{00}(\Gamma)\), we consider the following boundary value problem 
\begin{equation}
    \label{eq:modelProblem}
    \begin{aligned}
         - \div(A \grad u) + \lambda^2 u
            &= 0, \quad \text{in } \fulldomain,\\
        \restr{u}{\Gamma} 
            &= g,\quad \text{on } \Gamma,\\
        \restr{u}{\Gamma^C}
            &= 0,\quad \text{on } \Gamma^C.
    \end{aligned}
\end{equation}
Recall that \(\widetilde{g}\) is the extension of \(g\) by zero.  
As discussed in the previous section, we note that \(\widetilde{g} \in H^{1/2}(\dfulldomain)\).

We emphasize that we allow in general \(\Gamma = \partial \fulldomain\).
Moreover, the case \(\lambda = 0\) is permitted. 
However, to ensure coercivity of the problem, we restrict ourselves to the following two cases.
\begin{Assumption}\label{ass:DecayAssumption}
    Either \(\lambda > 0\) or \(\Gamma^C \neq \emptyset\).
\end{Assumption}

In the following, we briefly discuss the well-posedness of \eqref{eq:modelProblem}.
Since we have \(\widetilde{g}\in H^{\half}(\dfulldomain)\), we find a lifting $u_g\in H^1(\fulldomain)$ such that $\restr{u_g}{\dfulldomain} = \widetilde{g}$ with
\begin{equation}
    \label{eq:StabilityDiscreteLifting}
        \norm{u_g}_{H^1(\fulldomain)} \leq C_{\gamma} \|\widetilde{g}\|_{H^{\half}(\dfulldomain)} = C_{\gamma} \norm{g}_{H^{\half}_{00}(\Gamma)}, 
\end{equation} 
where \(C_{\gamma} > 0\) is constant, uniform in \(g\).
We introduce the continuous bilinear form
\begin{equation*}
    \widehat{a}_\Omega\,:\, H^1(\Omega)\times H^1(\Omega) \to \R,\quad
    (u,v) \mapsto \int_{\Omega} (A\nabla u\cdot) \nabla v\,\mathrm{d}x,
\end{equation*}
and the restriction 
\begin{equation*}
    a_\Omega = \restr{\widehat{a}_\Omega}{H^1_0(\Omega)\times H^1_0(\Omega)} .
\end{equation*}
The bilinear form induces an equivalent semi-norm on \(H^1(\Omega)\) defined as
\begin{equation*}
    \abs{u}_{a,\Omega}^2  = \widehat{a}_{\Omega}(u,u).
\end{equation*}
Thus, we define the natural energy norm of our problem as
\begin{equation}\label{eq:NaturalEnergyNorm}
    \energynorm{u}_\Omega^2 = \abs{u}_{a,\Omega}^2 + \lambda^2 \norm{u}_{L^2(\Omega)}^2.
\end{equation}
The variational formulation of problem \eqref{eq:modelProblem} reads: Seek \(u \in H^1(\fulldomain)\) such that
\(u_0 = u - u_g \in H^1_0(\fulldomain)\) and
\begin{equation}
    \label{eq:variationalFormulationAnalytic}
    \abil[\fulldomain]{u_0,v} + \lambda^2 \ip[\fulldomain]{u_0,v} = b(v), \quad \text{for all } v\in H^1_{0}(\fulldomain),
\end{equation}
with the continuous linear form \(b(v) = -\widehat{a}_{\fulldomain}(u_g,v) - \lambda^2(u_g,v)_{\fulldomain}\).
Note that the left-hand side of \eqref{eq:variationalFormulationAnalytic} is coercive.
The Lax-Milgram lemma, see \cite[Lem.~25.2]{ErnG21II}, thus shows that \eqref{eq:variationalFormulationAnalytic} is well-posed, and we obtain the \textit{a priori} estimate \(\energynorm{u_0}_{\Omega}\leq \energynorm{u_g}_{\Omega}\).
Furthermore, we obtain the estimate for the solution 
\begin{equation}\label{eq:APrioriEstimateSolution}
    \energynorm{u}_{\Omega} 
     \leq 2 \max\monosetc{\beta^{\half},\lambda} \norm{u_g}_{H^1(\Omega)} \leq  2 C_{\gamma} \max\monosetc{\beta^{\half},\lambda} \norm{g}_{H^{\half}_{00}(\Gamma)}.
\end{equation}

\subsection{Space discretization by conforming finite elements}
\label{subsec:FEM}
Next, we discretize the variational problem by standard conforming Lagrange finite elements.
We consider an affine, shape-regular, matching simplicial mesh $ \Th = \Th[\fulldomain] $ of $ \fulldomain $, see, e.g., \cite[Definition 8.11]{ErnG21I}, where the parameter $h$ denotes the maximal diameter of the elements in $\Th$. We denote the minimal diameter of the elements in $\Th$ by $\hmin$.
Let \(K\in \Th\) be an element.
We denote the set of all polynomials in $d$ variables of degree $\leq k$ by $\PkK$ for $k\geq 1$.
Based on these sets of polynomials we introduce the finite-dimensional %
approximation space
\begin{equation}
          \PkTh \coloneqq \big\{  v_h \in L^1(\fulldomain; \R) \quad | \quad  \restr{v_h}{K} \in \PkK, \; \forall K \in \Th \big\} \cap C^0(\fulldomain; \R).
\end{equation}
The subspace with homogenous boundary traces is denoted by \(\PkZeroTh\).

Since the treatment of inhomogeneous boundary data is essential to our problem, we briefly recall it in the following and refer to \cite{ErnG21II} for more details.
A finite element method is only able to produce solutions with boundary traces in the discrete trace space \(\restr{\PkTh}{\dfulldomain}\).
Therefore, the boundary data needs to be approximated, i.e., \(g\approx g_h\) for a suitable \(g_h\in\restr{\PkTh}{\dfulldomain}\).
For the approximate \(g_h\) we choose, like in the continuos case, a lifting
\begin{equation}\label{eq:AbstractDiscreteLifiting}
    u_{hg} \in \PkTh[],\qquad \text{s.t. } \restr{u_{hg}}{\dfulldomain} = g_h.
\end{equation}

The discretized analog of \eqref{eq:variationalFormulationAnalytic} then reads: Seek \(u_h\in \PkTh\) such that \(u_{h0} = u_h - u_{hg} \in \PkZeroTh\) and
\begin{equation}
    \label{eq:variationalFormulationDiscrete}
    \abil[\fulldomain]{u_{h0},v_h} + \lambda^2 (u_{h0},v_h)_\fulldomain 
        = b_h(v_h), \quad \text{for all } v_h\in \PkZeroTh,
\end{equation}
where \(b_h(v_h) = -\widehat{a}_\Omega(u_{hg},v_h) - \lambda^2 (u_{hg},v_h)_{\fulldomain}\).

As in the continuous setting, coercivity ensures well-posedness via the Lax-Milgram lemma, and we obtain analogously to \eqref{eq:StabilityDiscreteLifting}
\begin{equation}\label{eq:APrioriEstimateDiscreteSolution}
    \energynorm{u_h}_{\Omega} 
    \leq 2\bigenergynorm{u_{hg}}_{\Omega} 
    \leq 2 \max\monosetc{\beta^{\half},\lambda} \bignorm{u_{hg}}_{H^1(\Omega)}.
\end{equation}

A particular choice of a lifting \(u_{hg}\) in \eqref{eq:AbstractDiscreteLifiting} is a nodal one.
We briefly recall its construction.
For each element \(K \in \Th\), we introduce the nodal basis
\begin{equation*}
    \mathcal{N}_K \coloneqq \monosetc{ \sigma_{K,1}, \dots, \sigma_{K,N_k^d} }
\end{equation*} consisting of the point evaluations $\sigma_{K,i}(v) = v(a_{K,i})$ at suitable nodal points $a_{K,i} \in K$ for $i=1,\dots, N_k^d$, see e.g. \cite[Section 7.4]{ErnG21I}.
Hence, $(K,\PkK,\mathcal{N}_K)$ is the standard Lagrange finite element.
We denote the set of all nodal points in $\Th$ by $\mathcal{A}_h \coloneqq  \{a_1,\dots,a_{N_h}\}$ and 
corresponding global shape functions by $\varphi_a$ for $a \in \mathcal{A}_h$, as in \cite[Def. 19.2]{ErnG21I}.
The set of nodal points on the boundary is defined as \(\mathcal{A}_h^\partial = \setc{a\in \mathcal{A}_h}{\restr{\varphi_a}{\dfulldomain} \neq 0}\).

Let \(g \in H^{s^\ast}_0(\Gamma)\), for \(s^\ast > (d-1)/2\).
This ensures that the embedding \(H^{s}_0(\Gamma) \hookrightarrow C^0(\overline{\Gamma})\) is continuos, hence the pointwise evaluation of \(g\) is well-defined.
Thus, the nodal approximation of \(\widetilde{g}\) at the boundary is given by
\(
    g_h \coloneqq \sum_{a\in \mathcal{A}_h^\partial} \widetilde{g}(a) \restr{\varphi_a}{\dfulldomain},
\)
and the canonical lifting by
\begin{equation}
    \label{eq:def-nodal-lifting}
    u_{hg} \coloneqq \sum_{a\in \mathcal{A}_h^\partial} \widetilde{g}(a) \varphi_a \in  \PkTh,
    \quad  
\end{equation}
with the property that $ \bigrestr{u_{hg}}{\dfulldomain} =  g_h$.

While this nodal lifting is the natural choice for implementation, it is often desirable for analysis to introduce a discrete lifting that requires weaker regularity assumptions on g, and satisfies an estimate of the form \eqref{eq:StabilityDiscreteLifting}.
In Appendix~\ref{sec:LocalBoundaryLifting}, we construct such a lifting, which is later used for the domain decomposition application in Section~\ref{sec:Applications}.
Note, that the main results presented in Section~\ref{sec:MainResults} do not depend on the particular choice of the boundary lifting $u_{hg}$. This allows to pick the construction of $u_{hg}$ dependent on the requirements driven by the application.

\subsection{Patches \& patch energy}
In the following we introduce layered patches and associated energies.
Let $\Dh \subset \Th$ be a fixed subset of elements.
The set \(\Dh\) is called \textit{suitable} if $\Gamma \cap \overline{K} = \emptyset$ for all $K\in \Dh$, hence $\Dh$ does not touch the inhomogeneous part of the boundary.

We define patches of elements around $ \Dh $ recursively for \(\ell \in \N\) by 
\begin{equation}
    \label{eq:defPatches}
    \begin{aligned}
        \Patch{0}{\Dh} &\coloneqq \{  K  \; | \;  K \in \Dh\} = \Dh, \\
        \Patch{\ell}{\Dh} &\coloneqq \{ K \in \Th \; | \; \exists \widehat{K} \in \Patch{\ell -1}{\Dh} \; : \; \overline{K} \cap \overline{\widehat{K}} \neq \emptyset\}.
    \end{aligned}
\end{equation}
Patches are nested by definition, i.e., $\Patch{\ell}{\Dh} \subset \Patch{\ell + 1}{\Dh} $.
We call $\Patch{\ell}{\Dh}$ the patch around $\Dh$ of size $\ell$ and define the maximal suitable patch size as
\begin{equation}\label{eq:MaxPatchSize}
    \ellmax = \ellmax (\Dh) := \max \setc{\ell \in \N}{\forall K \in \Patch{\ell}{\Dh}:\,  \overline{K} \cap \supp u_{hg} = \emptyset}.
\end{equation}
Thus, \(\ellmax\) is the maximal number of layers separating the set \(\Dh\) from the support of the discrete lifting.
For a nodal lifting, as described in \eqref{eq:def-nodal-lifting}, the support of \(u_{hg}\) consists of one boundary layer around $\Gamma$.
Other discrete liftings might have a larger support, which we account for in \eqref{eq:MaxPatchSize}.
We refer to \Cref{Fig:Patches} for an illustration of the construction.

Finally, the \(\ell\)-th boundary layer of cells is denoted with 
\begin{equation}\label{eq:BoundaryLayer}
    \bLayer{\ell}(\Dh) \coloneqq \Patch{\ell}{\Dh} \setminus \Patch{\ell - 1}{\Dh}.
\end{equation}

\begin{figure}
    \begin{center}
        \includegraphics{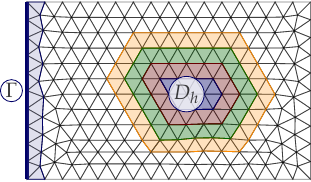}
    \end{center}
    \caption{First three patches $\Patch{\ell}{\Dh}$ for $\ell = 1$ (red), $\ell = 2$ (green) and $\ell = 3$ (orange). The support of nodal lifting $u_{hg}$ \eqref{eq:def-nodal-lifting} around $\Gamma$ is depicted in light-blue, and \(\ellmax = 7\).}
    \label{Fig:Patches}
\end{figure}

Let \(u_h = u_{h0} + u_{hg}\in \PkTh\) such that \(u_{h0} \in \PkZeroTh\) solves the discrete problem~\eqref{eq:variationalFormulationDiscrete}. 
For any $\ell \in \N$ we define the \textit{discrete energy in a patch around $\Dh$} by 
\begin{equation}
    \label{eq:DefEnergyOnPatches}
    E_{\Dh}(\ell) \coloneqq \sum_{K \in \Patch{\ell}{\Dh}}  \int_K A \grad u_h \cdot \grad u_h \dint{x} + \lambda^2 \int_K u_h u_h \dintx .
\end{equation}
This quantity is the discrete counterpart of the continuous energy functional introduced in \Cref{subsec:analyticalMotivation} and will be the central object in the discrete Saint-Venant estimate.

\bigskip

\section{Main Results}
\label{sec:MainResults}

We are now in a position to state the main result of this work.
It establishes a discrete Saint-Venant-type principle for conforming finite elements.
More precisely, we show that the discrete energy contained in successive element patches around a suitable interior cell set contracts by a factor strictly smaller than one.
This layer-wise contraction implies exponential decay of the discrete energy with respect to the number of mesh layers separating the patch from the inhomogeneous boundary.

\begin{Theorem}
    \label{Thm:discreteSaintVenant}
    Let \Cref{ass:DecayAssumption} hold and $0\leq\ell < \ellmax$. 
    \begin{enumerate}[a),align=left,leftmargin=\parindent,labelwidth=\parindent,labelsep=*]
        \item \label{Thm:discreteSaintVenant:A} If $\lambda > 0$, then
        \begin{equation}
            \label{Thm:discreteSaintVenant:A:equation}
            E_{\Dh}(\ell) \leq \frac{M_{\lambda}}{1+ M_{\lambda}} E_{\Dh}(\ell + 1) \quad \text{ with } \quad 
            M_{\lambda}= C_1 \bigl(1+\frac{1}{2\lambda \hmin}\bigr).
        \end{equation}
          
        \item \label{Thm:discreteSaintVenant:B} If $\lambda = 0$, then
        \begin{equation}
            \label{Thm:discreteSaintVenant:B:equation}
            E_{\Dh}(\ell) \leq \frac{M_0}{1+ M_0} E_{\Dh}(\ell + 1) 
            \quad \text{ with } \quad  
            M_0 = C_2 \bigl(1+\frac{1}{\hmin}\bigr).
        \end{equation}
      \end{enumerate}
    The constants $C_1,C_2>0$ only depend on the bounds of the material coefficient $\alpha, \beta$, and the shape-regularity of the mesh but are uniform in the mesh size and $\lambda$.
\end{Theorem}

The proof of \Cref{Thm:discreteSaintVenant} will be presented in \Cref{sec:ProofMainResults}.

\begin{Corollary}
    \label{cor:decayfactor}
    Let $\Dh\subset \Th$ be a suitable subset of elements. Then 
    \begin{equation}
        \label{eq:DecayFactor-rho}
        E_{\Dh}(\ell) \leq \rho E_{\Dh}(\ell+1), \quad 0 \leq \ell < \ellmax,
    \end{equation}
    with $\rho < 1$ depending on the regularity of the mesh, on the material parameters $\alpha$ and $\beta$, and also \(\lambda\) and the minimal element size \(\hmin\).
    Moreover, we obtain
    \begin{equation}
        \label{eq:DecayAndApriori}
        E_{\Dh}(0) 
        \leq \rho^{\ellmax} E_{\Dh}(\ellmax) 
        \leq  \rho^{\ellmax} \energynorm{u_h}^2_{\fulldomain} \; .
    \end{equation}
\end{Corollary}
    In words this means: The more layers of cells we can form between $\Dh$ and the support of $u_{hg}$, the smaller the discrete energy in $\Dh$. The scaling is exponential in the number of layers $\ellmax$.
\begin{proof}[Proof of Cor.~\ref{cor:decayfactor}]
    Following \Cref{Thm:discreteSaintVenant} the contraction factor $\rho$ in \eqref{eq:DecayFactor-rho} is given by 
    \begin{equation*}
        \rho = \frac{M}{1 + M} < 1,
    \end{equation*}
    with $M = M_\lambda$ for $\lambda > 0$ and $M=M_0$ for $\lambda = 0$ defined in \eqref{Thm:discreteSaintVenant:A:equation} and \eqref{Thm:discreteSaintVenant:B:equation}, respectively.
    To obtain $\eqref{eq:DecayAndApriori}$ we apply \Cref{Thm:discreteSaintVenant} $\ellmax$ times and then use \eqref{eq:APrioriEstimateDiscreteSolution}. 
\end{proof}

Next, we show that our decay rates agree with the expected analytic decay rate in the limit \( h\to 0^+\). In that sense, we show that the discrete contraction factor obtained above is asymptotically consistent with the classical continuous decay rate as the mesh is refined. In particular, under quasi-uniform refinement, the discrete layer-wise decay converges to an exponential decay in the physical distance to the boundary.

\begin{Corollary}\label{cor:asymp}
    Let the assumptions of \cref{Thm:discreteSaintVenant} hold.
    Assume that the mesh sequence is quasi-uniform, i.e., 
    \(\hmin/h \to \theta \in (0,1]\), for \(h\to 0^+\).
    \begin{enumerate}[a),align=left,leftmargin=\parindent,labelwidth=\parindent,labelsep=*]
        \item \label{Thm:discreteSaintVenant:A:asymp} 
        If $\lambda > 0$, then
        \begin{equation}
            \Bigl(\frac{M_{\lambda}}{1+M_{\lambda}}\Bigr)^{1/h}
                \to \exp\Bigl({-\frac{2\lambda \theta}{C_1}}\Bigr),\qquad
                h \to 0^+.
        \end{equation}
        \item \label{Thm:discreteSaintVenant:B:asymp} 
        If $\lambda = 0$, then
        \begin{equation}
            \Bigl(\frac{M_{0}}{1+M_{0}}\Bigr)^{1/h}
                \to \exp\Bigl({-\frac{\theta}{C_2}}\Bigr),\qquad
                h \to 0^+.
        \end{equation}
    \end{enumerate}
\end{Corollary}
\begin{proof}[Proof of Cor.~\ref{cor:asymp}]
    Let \(\rho = \frac{M_{\lambda}}{1+M_{\lambda}}\).
    We only prove \labelcref{Thm:discreteSaintVenant:A:asymp} since \labelcref{Thm:discreteSaintVenant:B:asymp} follows analogously with different coefficients.
    We determine the leading \(\hmin\)-coefficient of the map \(\hmin\mapsto 1- \rho\) as 
    \begin{equation*}
        1 - \rho 
            = \frac{2\lambda \hmin}{ C_1} \frac{1}{1 + 2\lambda(C_1^{-1}+ 1) \hmin} .
    \end{equation*}
    Therefore, 
    \begin{equation*}
         \frac{1}{\hmin}(1-\rho) \to \frac{2\lambda}{C_1},\quad \text{for} \quad \hmin\to 0^+,
    \end{equation*}
    where we used the expansion \((1+\alpha \hmin)^{-1} = 1 + \mathcal{O}(\hmin)\), \(\hmin\to 0^+\), for any \(\alpha > 0\) in the second term. 
    Thus, the expansion yields
    \begin{equation*}
        \rho = 1 - \frac{2\lambda \hmin}{C_1} + \mathcal{O}(\hmin^2),\qquad \hmin\to 0^+.
    \end{equation*}
    With the expansion
    \(
        \ln(1+x) = x + \mathcal{O}(x^2), \quad \text{for} \quad |x|\to 0,
    \)
    we obtain
    \begin{equation*}
        \ln(\rho^{1 /\hmin})
            = \frac{1}{\hmin} \ln(\rho)
            = -\frac{2\lambda}{C_1} + \mathcal{O}(\hmin) \to -\frac{2\lambda}{C_1},\qquad \hmin\to 0^+.
    \end{equation*}
    Finally, the claim is a result of the quasi-uniformity. 
\end{proof}

\begin{figure}
    \begin{center}
        \includegraphics{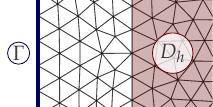}
        \includegraphics{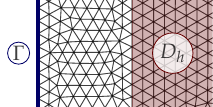}
    \end{center}
    \caption{An interior set $\Dh$ (red) is separated from the inhomogeneous boundary $\Gamma$ (blue) by several layers of mesh elements (white).
    Left: coarse mesh with relatively few admissible layers.
    Right: refined mesh with smaller mesh size $h$, allowing more layers between $\Dh$ and $\Gamma$.
    For fixed geometric distance $\delta = \inf_{x\in \Dh} \operatorname{dist}(x,\Gamma)$, the number of layers scales as $\ell_{\max} \sim \delta/h$.}
    \label{Fig:GeometricInterpretation}
\end{figure}

\begin{Remark}[Geometric interpretation]\label{Rem:Interpretation-rho-scaling}
    Corollary~\ref{cor:asymp} allows for a qualitative geometric interpretation.
    Given an initial patch \(\Dh\) with distance \(\delta = \inf_{x\in \Dh} \operatorname{dist}(x,\Gamma)\) to the boundary, the maximal number of possible layers \eqref{eq:MaxPatchSize} increases with the mesh size in the limit \(h\to 0^+\), i.e., 
    \(
        \ellmax \sim \delta / h.
    \)
    Figure~\ref{Fig:GeometricInterpretation} illustrates this relation between mesh layers and geometric distance.
    Hence, we obtain with Corollaries~\ref{cor:decayfactor} and \ref{cor:asymp} in the limit the estimate
    \begin{equation*}
        E_{\Dh}(0) 
            \lesssim \rho^{\ellmax} E_{\Dh}(\ellmax)
            \lesssim e^{- c \delta} \|g\|_{H^{s^\ast}(\Gamma)}^2,
    \end{equation*}
    where $c = 2\theta\lambda / C_1$ for $\lambda > 0$, and 
    $c = \theta/ C_2$ for $\lambda = 0$, respectively.
    Thus, the energy norm of the solution decays exponentially in the distance of the patch to the boundary.
    This is precisely a discrete version of the underlying analytic principle.
\end{Remark}

\bigskip 

\section{Smooth cutoff on patches}
\label{sec:Cutoff}
To evaluate the discrete energy locally on patches, we need a mechanism that allows us to restrict the discrete solution to a given patch without leaving the finite element space.
To this end, we multiply the solution with a suitable cutoff function and interpolate the resulting product back into $\PkTh$.
In order to use this construction in the analysis, we require stability of the nodal interpolation operator when applied to products of discrete functions.
We therefore briefly recall the definition of the Lagrange interpolation operator and collect the stability property needed in the sequel.

\subsection{Nodal interpolation}

Recall from \cref{subsec:FEM} the definition of the Lagrange element 
$\bigl(K, \PkK, \mathcal{N}_K\bigr)$.
Let $v \in C(K, \R)$.
We define the local interpolant by
\begin{equation*}
    \label{eq:LocalInt}
    \LocInt v =  \sum_{i=1}^{N_k^d} \sigma_{K,i}(v) \varphi_{a_{K,i}} = \sum_{i=1}^{N_k^d} v(a_{K,i}) \varphi_{a_{K,i}},  
\end{equation*}
where $\varphi_{a_{K,i}}$ denotes the local shape function associated with the nodal point $a_{K,i}$.
We can then further retrieve the standard Lagrange nodal interpolation for $v\in C^0(\overline{\fulldomain})$ by 
\begin{equation}
    \label{eq:NodalInt}
    \restr{\big(\NodalInt v \big)}{K} = \LocInt v, \qquad \forall K \in \Th.
\end{equation}

\begin{Lemma}
    \label{Lem:StabilityIhProductIh}
    Let $K \in \Th$. For $m=0,1$ there exists a constant $\CInth$ uniform in $K$ and $h$, s.t.
    \begin{equation*}
        \abs{\NodalInt(\phi_h \psi_h)}_{H^{m}(K)} \leq \CInth \abs{\phi_h \psi_h}_{H^m(K)},
    \end{equation*}
    for all $\phi_h, \psi_h \in \PkTh$. The constant $\CInth$ depends on the polynomial degree $k$, the spatial dimension $n$ and the shape regularity of the mesh $\Th$.
\end{Lemma}
\begin{proof}
    Let $m=0$ or $m=1$ and $\abs{\cdot}_{H^{0}} = \norm{\cdot}_{L^2}$.
    The argument is carried out on the reference element $\widehat{K}$. 
    The map 
    \begin{equation*}
        Id - \LocIntRef : (\mathbb{P}_{2k}(\widehat{K}), \abs{\cdot}_{H^m}) \to (\mathbb{P}_{2k}(\widehat{K}), \abs{\cdot}_{H^m})
    \end{equation*}
    is linear and bounded since $\dim (\mathbb{P}_{2k}(\widehat{K})) < \infty$ and the constant functions are in the kernel of the operator.
    Thus, it holds for \(\widehat{\phi}_h,\widehat{\psi}_h \in  \mathbb{P}_{k}(\widehat{K})\)
    \begin{equation*}
        \abs{\NodalInt(\widehat{\phi}_h \widehat{\psi}_h)}_{H^{m}(\widehat{K})} 
            \leq \abs{\NodalInt(\widehat{\phi}_h \widehat{\psi}_h) - \widehat{\phi}_h \widehat{\psi}_h}_{H^{m}(\widehat{K})} + \abs{\widehat{\phi}_h \widehat{\psi}_h}_{H^{m}(\widehat{K})}
            \leq c_1 \abs{\widehat{\phi}_h \widehat{\psi}_h}_{H^{m}(\widehat{K})},
    \end{equation*}
    where \(c_1>0\) depends on \(k\), the reference element and the nodal set.
    Let $F_K: \widehat{K} \to K$ with $F_K(\widehat{x}) = B_K \widehat{x} + b_K$ be the affine map from the reference element $\widehat{K}$ to the physical element $K$.
    For any \(\phi_h,\psi_h\in \mathbb{P}_{k}(K)\), we define \(p = \phi_h\psi_h\) and \(\widehat{p} = p\circ F_K\).
    With the usual transformation properties, see, e.g., \cite[Theorem 3.1.2]{Cia02}, we conclude that
    \begin{align*}
        \abs{\NodalInt p}_{H^m(K)} 
        &\leq c_2 \bignorm{B^{-1}}^m \abs{\det B}^{\half} \abs{\LocIntRef \widehat{p}}_{H^m(\widehat{K})} \\
        &\leq c_2 \bignorm{B^{-1}}^m \abs{\det B}^{\half} c_1\abs{\widehat{p}}_{H^m(\widehat{K})} \\
        &\leq c_1 c_2^2 \bignorm{B}^m \bignorm{B^{-1}}^m  \abs{p}_{H^m(K)},
    \end{align*}
    with a constant \(c_2> 0\) depending on the dimension and \(m\). 
    Furthermore, the product \(\bignorm{B}\bignorm{B^{-1}}\) is bounded uniformly by the shape regularity, see, e.g., \cite[Theorem 3.1.3]{Cia02}. This proves the claim.
\end{proof}

\subsection{Discrete cutoff}

Recall the definition of patches $\Patch{\ell}{\Dh}$ in \eqref{eq:defPatches}. 
Similar to \cite[Definition 3.3]{MalP14} we define for each $\Patch{\ell}{\Dh}$ a discrete cutoff function $\Cutoff \in \PkTh$ %
with 
\begin{subequations}
    \label{eq:AssCutoff}
    \begin{alignat}{2}
         \label{eq:AssCutoff:A} 0 \leq \Cutoff &\leq 1, &&\text{on } \fulldomain, \\
        \label{eq:AssCutoff:B} \Cutoff &= 0, &&\text{on } \overline{K} \text{ for } K \in \Th \setminus \Patch{\ell+1}{\Dh}, \\
        \label{eq:AssCutoff:C} \Cutoff &= 1, &&\text{on } \overline{K} \text{ for } K \in \Patch{\ell}{\Dh}, \\
        \label{eq:AssCutoff:D} \bignorm{\grad \Cutoff}_{L^{\infty}(\fulldomain)} &\leq \CCutoff \hmin^{-1} \; ,
    \end{alignat}
\end{subequations}
with a constant $\CCutoff>1$ uniform in $\hmin$. 

By combining this cutoff function with the element-wise nodal interpolation $\NodalInt$ defined in \eqref{eq:NodalInt} we define the following discrete cutoff operator, which will play a crucial role in the proof of \Cref{Thm:discreteSaintVenant}.
    We define
    \begin{equation}
        \label{eq:DefIntCutoff}
        \IntCutoff: \PkTh \to \PkTh, \quad v_h \mapsto \NodalInt( \Cutoff v_h).
    \end{equation}

Next, we gather important properties of $\IntCutoff$ in the following Lemma. 

\begin{Lemma} 
    \label{Lem:PropCutoff}
    Let $v_h \in \PkTh$. Then, it holds:
    \begin{enumerate}[a),align=left,leftmargin=\parindent,labelwidth=\parindent,labelsep=*]
        \item \label{Lem:PropCutoff:A} In $K \in \Patch{\ell}{\Dh}$,
        \begin{equation*}
            \label{eq:P1IntCutoff}
            \restr{\IntCutoff v_h}{K} = \restr{v_h}{K} \; .
        \end{equation*}
        \item \label{Lem:PropCutoff:B} In $K \in \Th \setminus \Patch{\ell + 1}{\Dh}$,
        \begin{equation*}
             \label{eq:P2IntCutoff}
            \restr{\IntCutoff v_h}{K} = 0 \; . 
        \end{equation*}
        \item \label{Lem:PropCutoff:C} In $K \in \bLayer{\ell+1}$, 
        \begin{equation*}
            \label{eq:P3IntCutoff}
            \int_{K} v_h \IntCutoff v_h \dintx 
            \leq \CInth \norm{v_h}_{L^2(K)}^2 \; . 
        \end{equation*}
        \item \label{Lem:PropCutoff:D} In $K \in \bLayer{\ell+1}$, 
        \begin{equation*}
            \label{eq:P4IntCutoff}
            \int_{K} A \grad v_h \cdot \grad (\IntCutoff v_h) \dintx
            \leq \CInth \beta^{\half} \norm{v_h}_{a,K} 
        \Bigl( \CCutoff \hmin^{-1} \norm{v_h}_{L^2(K)} + \mfrac{1}{\alpha^{\half}}\norm{v_h}_{a,K}  \Bigr) \; .
        \end{equation*}
    \end{enumerate}
\end{Lemma}
\begin{proof}
    Let  $K \in \Patch{\ell}{\Dh}$. 
    We have $\Cutoff = 1$ on $\overline{K}$ by \eqref{eq:AssCutoff:C}. 
    Since $v_h \in \PkTh$ and \eqref{eq:NodalInt} this yields
    \begin{equation*}
        \restr{\IntCutoff v_h}{K} = \restr{\NodalInt(\Cutoff v_h)}{K} = \LocInt(\restr{v_h}{K}) = \restr{v_h}{K}\;, 
    \end{equation*}
    which shows \cref{Lem:PropCutoff}\ref{Lem:PropCutoff:A}.
    Next, let $K \in \Th\setminus \Patch{\ell+1}{D_h}$. Then we get
        \begin{equation*}
            \restr{\IntCutoff v_h}{K} = \restr{\NodalInt(\Cutoff v_h)}{K} = \LocInt(0) = 0 \; , 
        \end{equation*}
    since $\Cutoff = 0 $ on $\overline{K}$ by \eqref{eq:AssCutoff:B}, so \cref{Lem:PropCutoff}\ref{Lem:PropCutoff:B} holds.

    For the remaining part of the proof, let $K\in \bLayer{\ell + 1}$.
    Note, that
    \begin{equation}
        \label{eq:estL2term}
        \int_K v_h (\IntCutoff v_h) \dintx 
        \leq \norm{v_h}_{L^2(K)} \bignorm{\NodalInt (\Cutoff v_h)}_{L^2(K)} \; .
    \end{equation}
    By \Cref{Lem:StabilityIhProductIh} and \eqref{eq:AssCutoff:A} we get
    \begin{equation*}
        \bignorm{\NodalInt (\Cutoff v_h)}_{L^2(K)} 
        \leq \CInth \bignorm{\Cutoff v_h }_{L^2(K)} 
        \leq \CInth \bignorm{\Cutoff}_{L^{\infty}(K)} \norm{v_h }_{L^2(K)} \leq  \CInth \norm{v_h}_{L^2(K)} .
    \end{equation*}
    Combining this with \eqref{eq:estL2term} shows \cref{Lem:PropCutoff}\ref{Lem:PropCutoff:C}.

    Similarly, using the bound on $\bignorm{\grad \Cutoff}_{L^{\infty}(\fulldomain)}$ from \eqref{eq:AssCutoff:D} and \Cref{Lem:StabilityIhProductIh} we derive
    \begin{align*}
        \int_K A \grad v_h \grad ( \IntCutoff v_h) \dintx 
        &\leq \norm{v_h}_{a,K} \norm{\NodalInt(\Cutoff v_h)}_{a,K} \\
        &\leq \beta^{\half} \norm{v_h}_{a,K} \abs{\NodalInt \big(\Cutoff v_h\big)}_{H^1(K)} \\
        &\leq \CInth \beta^{\half} \norm{v_h}_{a,K} \abs{\Cutoff v_h}_{H^1(K)}\; .
    \end{align*}
    Further, by the product rule we have 
    \begin{align*}
        \abs{\Cutoff v_h}_{H^1(K)} 
        &\leq \norm{\grad \Cutoff}_{L^{\infty}(K)} \bignorm{v_h}_{L^2(K)}  + \bignorm{\Cutoff}_{L^{\infty}(K)} \norm{\grad v_h}_{L^2(K)} \\
        &\leq \CCutoff \hmin^{-1} \norm{v_h}_{L^2(K)} + \mfrac{1}{\alpha^{\half}}\norm{v_h}_{a,K},
    \end{align*}
    which concludes the proof.
\end{proof}

\bigskip

\section{Proof of \texorpdfstring{\Cref{Thm:discreteSaintVenant}}{the main result}}
\label{sec:ProofMainResults}

This section contains the proof of \Cref{Thm:discreteSaintVenant}. 
The argument is structured in two steps.

First, we show that the energy $E_{\Dh}(\ell) $ on the patch $\Patch{\ell}{\Dh}$ defined in \eqref{eq:DefEnergyOnPatches} can be represented as a flux across the outer layer $\bLayer{\ell+1}$.
Second, we bound this flux by the difference between the energies on two
consecutive patches, $\Patch{\ell+1}{\Dh}$ and $\Patch{\ell}{\Dh}$.
Combining both steps yields a contraction estimate
\begin{equation}\label{eq:ToBeRearanged}
    E_{\Dh}(\ell) \leq M \bigl(E_{\Dh}(\ell+1) -  E_{\Dh}(\ell) \bigr),
\end{equation}
where the constant $M>0$ depends explicitly on the parameter $\lambda$,
the minimal mesh size $h_{\min}$, the bounds on the material coefficient
$\alpha, \beta$, and the shape-regularity of the mesh.
The proof tracks these dependencies throughout. 
We treat the cases $\lambda>0$ and $\lambda=0$ separately.
Rearranging \eqref{eq:ToBeRearanged} then gives \eqref{Thm:discreteSaintVenant:A:equation} and \eqref{Thm:discreteSaintVenant:B:equation}, respectively.

\subsection{Strictly positive \texorpdfstring{$\lambda > 0$}{parameter}}

\begin{proof}[Proof of \cref{Thm:discreteSaintVenant}\ref{Thm:discreteSaintVenant:A}]
    Recall, that $u_{h0} = u_h - u_{hg} \in \PkZeroTh$ satisfies \eqref{eq:variationalFormulationDiscrete}. By assumption it holds that 
    $\ell < \ellmax$, and the definition of $\ellmax$ in \eqref{eq:MaxPatchSize} implies that 
    $\Patch{\ell + 1}{\Dh} \cap \supp u_{hg} = \emptyset$.
    Therefore, the energy on \(\Patch{\ell}{\Dh}\) is the same for both \(u_h\) and \(u_{h0}\), i.e., we define
    \begin{equation}
        \label{eq:DefEnergyOnPatchesZeroLambda}
        E_{\Dh,0}(\ell) \coloneqq   \sum_{K \in \Patch{\ell}{\Dh}}  \int_K A \grad u_{h0} \cdot \grad u_{h0} \dint{x} + \lambda^2 \int_K u_{h0} u_{h0} \dintx = E_{\Dh}(\ell).
    \end{equation}
    Let \(\IntCutoff\) be defined as in \eqref{eq:DefIntCutoff}.
    By \cref{Lem:PropCutoff}\ref{Lem:PropCutoff:A}, it holds that \(\restr{\IntCutoff u_{h0}}{K} =  u_{h0}\) for \(K\in\Patch{\ell}{\Dh}\).
    Thus, we can plug the cut-off function under the integral and obtain
    \begin{equation}
        \label{eq:RespresentationEnergyZeroLambda}
        E_{\Dh,0}(\ell)  = \sum_{K \in \Patch{\ell}{\Dh}}  \int_K A \grad u_{h0} \cdot \grad ( \IntCutoff u_{h0}) \dint{x} + \lambda^2 \int_K u_{h0} (\IntCutoff u_{h0}) \dintx .
    \end{equation}

    Again, by definition, the function $\IntCutoff u_{h0} \in \PkZeroTh$ is a valid test function for the variational formulation \eqref{eq:variationalFormulationDiscrete}, and we see that 
    \begin{equation}
        \label{eq:WeakFormEnergyZeroLambda}
        \abil[\fulldomain]{u_{h0}, \IntCutoff u_{h0}} + \lambda^2 (u_{h0}, \IntCutoff u_{h0})_\fulldomain 
        = b_h(\IntCutoff u_{h0}).
    \end{equation}
    From \(b_h(\IntCutoff u_{h0}) = -\widehat{a}_\Omega(u_{hg},\IntCutoff u_{h0}) - \lambda^2 (u_{hg},\IntCutoff u_{h0})_{\fulldomain}\) we conclude that
        \begin{equation}
            \label{lem:rhsZeroIfNoSupport}
               \supp \IntCutoff u_{h0}  \cap \supp u_{hg} = \emptyset
                \implies
                b_h(\IntCutoff u_{h0}) = 0.
        \end{equation}
   Hence, the right-hand side of $\eqref{eq:WeakFormEnergyZeroLambda}$ vanishes.

    This allows for subtraction of
    \eqref{eq:WeakFormEnergyZeroLambda} from \eqref{eq:RespresentationEnergyZeroLambda}, which yields together with Lem-ma~\ref{Lem:PropCutoff}\ref{Lem:PropCutoff:B}
    \begin{equation}
        \label{eq:EnergyZeroIsBlayerFluxLambda}
        \begin{aligned}
            E_{\Dh,0}(\ell)
            =& \sum_{K \in \Patch{\ell}{\Dh}}  \int_K A \grad u_{h0} \cdot \grad ( \IntCutoff u_{h0}) \dint{x} + \lambda^2 \int_K u_{h0} (\IntCutoff u_{h0}) \dintx \\
            &- \abil[\fulldomain]{u_{h0}, \IntCutoff u_{h0}} + \lambda^2 (u_{h0}, \IntCutoff u_{h0})_\fulldomain \\
            =& \sum_{K \in \Patch{\ell}{\Dh}}  \int_K A \grad u_{h0} \cdot \grad ( \IntCutoff u_{h0}) \dint{x} + \lambda^2 \int_K u_{h0} (\IntCutoff u_{h0}) \dintx \\
            &- \sum_{K \in \Th}  \int_K A \grad u_{h0} \cdot \grad ( \IntCutoff u_{h0}) \dint{x} + \lambda^2 \int_K u_{h0} (\IntCutoff u_{h0}) \dintx  \\
            =& - \sum_{K \in \bLayer{\ell+1}} \int_K A \grad u_{h0} \cdot  \grad ( \IntCutoff u_{h0}) \dint{x} + \lambda^2 \int_K u_{h0} (\IntCutoff u_{h0}) \dintx \; .
        \end{aligned}
    \end{equation}

    By assumption \eqref{eq:MaxPatchSize}, the lifting \(u_{hg}\) is unsupported in \(\bLayer{\ell+1}\), i.e., $\overline{K} \cap \supp u_{hg} = \emptyset$ for $K \in \bLayer{\ell + 1} $.
    Therefore, \(\restr{u_{h}}{K} = \restr{u_{h0}}{K}\) for any \(K\in \bLayer{\ell +1}\), and the equality in \eqref{eq:RespresentationEnergyZeroLambda} shows that
    \begin{equation}
        \label{eq:EnergyIsBlayerFluxLambda}
        E_{\Dh}(\ell) = E_{\Dh,0}(\ell )= - \sum_{K \in \bLayer{\ell+1}} \int_K A \grad u_{h} \cdot  \grad ( \IntCutoff u_{h}) \dint{x} + \lambda^2 \int_K u_{h} (\IntCutoff u_{h}) \dintx \;.
    \end{equation}
    Using \cref{Lem:PropCutoff}\ref{Lem:PropCutoff:C} and \cref{Lem:PropCutoff}\ref{Lem:PropCutoff:D} we derive the estimate 
    \begin{equation*}
        \label{eq:BlayerFluxEstimate}
        E_{\Dh}(\ell)  \leq \CInth \!  \sum_{K \in \bLayer{\ell+1}} \! \Big( \lambda ^2 \norm{u_h}_{L^2(K)}^2 \! +  \beta^{\half} \norm{u_h}_{a,K} \!
        \Bigl(\CCutoff \hmin^{-1} \norm{u_h}_{L^2(K)} + \mfrac{1}{\alpha^{\half}}\norm{u_h}_{a,K}  \Bigr) \! \Big) .
    \end{equation*}
    The occurring mixed term is estimated by Young's inequality 
    \begin{equation}\label{eq:UpperEstimateLambda}
        \sum_{K \in  \bLayer{\ell+1}}  \norm{u_h}_{a,K} \norm{u_h}_{L^2(K)}
        \leq \frac{1}{2 \lambda} \Bigl( \sum_{K \in  \bLayer{\ell+1}} \norm{u_h}_{a,K}^2  + \lambda^2 \sum_{K \in  \bLayer{\ell+1}}   \norm{u_h}_{L^2(K)}^2 \Bigr),
    \end{equation}
    which yields for \(\CCutoff > 1\)
    \begin{equation}\label{eq::EnergyPlBoundByDifferencePre}
         E_{\Dh}(\ell) 
         \leq \CInth \Bigl( \frac{\beta^{\half}}{\alpha^{\half}} +  \frac{ \beta^{\half}\CCutoff}{2\lambda \hmin} \Bigr) \Bigl( \sum_{K \in  \bLayer{\ell+1}} \norm{u_h}_{a,K}^2  + \lambda^2 \sum_{K \in  \bLayer{\ell+1}}   \norm{u_h}_{L^2(K)}^2 \Bigr).
    \end{equation}
    The right-hand side term in \eqref{eq::EnergyPlBoundByDifference} is now the difference between the energies, i.e.,
    \begin{equation}\label{eq:EnergyDifferences}
        \sum_{K \in  \bLayer{\ell+1}} \norm{u_h}_{a,K}^2  + \lambda^2 \sum_{K \in  \bLayer{\ell+1}}   \norm{u_h}_{L^2(K)}^2  = E_{\Dh}(\ell+1) - E_{\Dh}(\ell).
    \end{equation}
    Hence, this proves the claim
    \begin{align}
        \label{eq::EnergyPlBoundByDifference}
         E_{\Dh}(\ell)
        &\leq C_1 \big( 1 + \frac1{2\lambda \hmin} \big)\Bigl( E_{\Dh}(\ell+1) - E_{\Dh}(\ell)  \Bigr),
    \end{align}
    where the constant $C_1>0$ depends on $\CInth, \CCutoff, \alpha$, and $\beta$.
\end{proof}

\subsection{Vanishing \texorpdfstring{$\lambda = 0$}{parameter} }
We proceed with the proof for the case $\lambda = 0$, which is similar to the previous one. 
Additionally, we require a Poincaré inequality, e.g. \cite[Thm.~10.6.12]{BreS08}. 

\begin{proof}[Proof of \cref{Thm:discreteSaintVenant}\ref{Thm:discreteSaintVenant:B}]
    We distinguish again two cases: First, let $\overline{K}\cap \Gamma^C = \emptyset$ for all $K \in \Patch{\ell}{\Dh}$.
    Let $c\geq 0$ be a constant, which will be determined later.
    Since $\grad c = 0$, we proceed as in the proof of \Cref{Thm:discreteSaintVenant} and obtain
    \begin{equation}
        \label{eq:EnergyPlTwo}
        \begin{aligned}
            E_{\Dh,0} 
            \coloneqq \sum_{K \in \Patch{\ell}{\Dh}}  \int_K A \grad u_{h0} \grad (u_{h0}-c) \dint{x}
            = 
            E_{\Dh}(\ell)
        \end{aligned}
    \end{equation}
    With \cref{Lem:PropCutoff}\ref{Lem:PropCutoff:A} it holds that
    \begin{equation}\label{eq:NullEnergyEstimate}
        E_{\Dh,0} = \sum_{K \in \Patch{\ell}{\Dh}}  \int_K A \grad u_{h0} \grad (\IntCutoff(u_{h0}-c)) \dint{x}.
    \end{equation}
    Since $\overline{K}\cap \dfulldomain = \emptyset$ for all $K \in \Patch{\ell}{\Dh}$, the function $\IntCutoff(u_{h0} - c) \in \PkZeroTh$ is a valid test function and \eqref{eq:variationalFormulationDiscrete} shows that
    \begin{equation}
        \label{eq:WeakFormEnergyZeroNoLambda}
        \abil[\fulldomain]{u_{h0}, \IntCutoff (u_{h0}-c)}
        = b_h\bigl(\IntCutoff (u_{h0}-c)\bigr) = 0.
    \end{equation}
    The last equality is a consequence of $K \cap \supp u_{hg} = \emptyset$ for $K \in \bLayer{\ell + 1} $ since $\ell < \ellmax$. 
    
    Taking the difference between \eqref{eq:NullEnergyEstimate} and \eqref{eq:WeakFormEnergyZeroNoLambda} and repeating arguments above for \eqref{eq:EnergyZeroIsBlayerFluxLambda} and \eqref{eq:EnergyIsBlayerFluxLambda} shows that
    \begin{equation}
        \label{eq:FluxThroughInterfaceNoLambda}
        E_{\Dh}(\ell) = - \sum_{K \in  \bLayer{\ell+1}} \int_K A \grad u_{h} \grad ( \IntCutoff (u_{h}-c)) \dint{x}.
    \end{equation}
    Using \cref{Lem:PropCutoff}\ref{Lem:PropCutoff:D} and  Young's inequality with $\gamma > 0$ shows that
    \begin{equation}
        \label{eq:estimateFluxInterfaceNoLambda}
        E_{\Dh}(\ell) \leq \CInth \beta^{\half} \Bigl( \sum_{K \in \bLayer{\ell + 1}}  \! \bigl( \frac{1}{\alpha^{\half}} + \frac{\CCutoff}{2\gamma \hmin}  \bigr) \! \norm{u_h}^2_{a,K} + \frac{\gamma \CCutoff}{2 \hmin} \int_{Z_{\ell + 1}}\!|u_h - c|^2\dint{x}\Bigr), \!
    \end{equation}
    where $Z_{\ell+1} = \interior \bigl( \bigcup_{K \in \bLayer{\ell + 1}} \overline{K} \bigr)$.
    We choose \(c=\int_{Z_{\ell+1}} u_h \dintx\) and obtain with the Poincaré inequality, see for example \cite[Thm.~10.6.12, eq. (10.6.13)]{BreS08}, that
    \begin{equation}
        \label{eq:Poincare}
        \int_{Z_{\ell+1}} |u_h -c|^2 \dint{x}
        \leq C_{\mathrm{p}} \sum_{K \in \bLayer{\ell +1 }} |u_h|_{H^1(K)}^2 
        \leq \frac{C_{\mathrm{p}}}{\alpha} \sum_{K \in \bLayer{\ell +1 }} \norm{u_h}_{a,K}^2 ,
    \end{equation}
    with a constant \(C_{\mathrm{p}}>0\) that depends on the smallest angle of all elements \(K\) comprising \(\bLayer{\ell+1}\). 
    Thus, from \eqref{eq:estimateFluxInterfaceNoLambda} we obtain
    \begin{equation*}
        E_{\Dh}(\ell) 
        \leq \CInth \frac{\beta^{\half}}{\alpha^{\half}} \Bigl(  1 + \frac{\CCutoff \alpha^{\half}}{2 \gamma \hmin} + \frac{\gamma \CCutoff C_{\mathrm{p}}}{2 \hmin \alpha^{\half}}\Bigr)
        \sum_{K\in \bLayer{\ell+1}} \norm{u_h}^2_{a,K},
    \end{equation*}
    which yields by choosing \(\gamma = \bigl( \alpha / C_{\mathrm{p}} \bigr)^{\half}\) and repeating the argument from above for \eqref{eq:UpperEstimateLambda}-\eqref{eq:EnergyDifferences} 
    \begin{equation}
        E_{\Dh}(\ell) \leq \CInth \frac{\beta^{\half}}{\alpha^{\half}} \Bigl(  1 +  \frac{ \CCutoff  \sqrt{ C_{\mathrm{p}}}}{\hmin}\Bigr)  \bigl( E_{\Dh}(\ell+1) - E_{\Dh}(\ell)\bigr). 
    \end{equation}
    The claim follows for this case directly with $C_2 = \CInth \CCutoff \beta^{\half} \alpha^{-\half} \sqrt{C_{\mathrm{p}}}$.

    In the second case, the patch $\Patch{\ell}{\Dh}$ touches the homogeneous Dirichlet boundary, i.e., $\overline{K}_{\star} \cap \Gamma^C \neq \emptyset$ for some $K_{\star} \in \Patch{\ell}{K}$. 
    Here, we set $c = 0$ in above derivation and replace \eqref{eq:Poincare} using \cite[Thm.~10.6.12, eq. (10.6.14)]{BreS08} with 
    \begin{equation}
        \label{eq:PoincareZero}
        \norm{u_h}_{L^2(Z_{\ell + 1})}^2
        \leq \frac{C_{\mathrm{p},0}}{\alpha} \sum_{K \in \bLayer{\ell +1 }} \norm{u_h}^2_{a,K}, 
    \end{equation}
    where $C_{\mathrm{p},0}$ depends only on the shape regularity of $\bLayer{\ell + 1}$.
    The rest of the proof follows analogously with $C_2 = \CInth \CCutoff \beta^{\half} \alpha^{-\half} \sqrt{C_{\mathrm{p},0}}$.
\end{proof}

\begin{Remark}[Localized right-hand sides]
The mechanism behind \Cref{Thm:discreteSaintVenant} relies on spatial
separation between inhomogeneous data and the region where the energy is
measured rather than on boundary conditions themselves. Consequently,
analogous decay estimates hold for problems with homogeneous boundary
conditions and spatially localized right-hand sides, since the discrete
solution satisfies a homogeneous equation outside the support of the data.

This recovers exponential energy decay estimates used in localized
orthogonal decomposition and related multiscale methods; see, e.g.,
\cite[Theorem~4.1]{MalP21}, \cite[Corollary 3.7]{AltHP20}, and \cite[Lemma~3.1]{GalMa23}. Hence, boundary-
and source-driven decay arise from the same locality property of the
discrete elliptic operator.
\end{Remark}

\begin{Remark}
While the construction of the discrete cutoff operator from \Cref{sec:Cutoff} reminds of techniques used in the analysis of multiscale methods — in particular in the localization of corrector functions within the LOD framework, see for example \cite{MalP14} and \cite[Chapter~4]{MalP21} — the present approach differs in a fundamental way.
No multiscale decomposition or orthogonality property is invoked.
Instead, the localization mechanism is applied directly to the full conforming finite element solution, while using the discrete product estimate of the nodal interpolation.
From a conceptual point of view, the proof of \Cref{Thm:discreteSaintVenant} presented in \Cref{sec:ProofMainResults} is closer to the classical energy-based decay arguments of Mieth and Horgan, translated to the discrete setting.    
\end{Remark}

\bigskip

\section{Application: Discrete Parallel Schwarz Method}
\label{sec:Applications}

The discrete decay estimate of \Cref{Thm:discreteSaintVenant} provides a
natural tool for studying the propagation of interface information in
iterative domain decomposition methods. In particular, it allows us to
analyze the convergence of a parallel Schwarz iteration for a finite
element discretization of an elliptic boundary value problem.

To this end, we consider the Laplace problem on a bounded polygonal domain $\fulldomain \subset \R^2$,
\begin{equation}
    \begin{aligned}
        \label{eq:globalProblemParallelSchwarz}
        - \laplace u &= f,  &&\text{in } \fulldomain, \\
        u &= 0 , &&\text{on } \dfulldomain.
    \end{aligned}
\end{equation}
Rather than solving \eqref{eq:globalProblemParallelSchwarz} directly on $\fulldomain$, we introduce overlapping subdomains $\fulldomain_1$ and $\fulldomain_2$, together with their interfaces 
\[
    \Gamma_i \coloneqq \dfulldomain_i \cap \fulldomain, \qquad i = 1,2.
\]
We then solve iteratively, for $n \geq 0$,
\begin{equation}
    \label{eq:SubproblemsParallelSchwarz}
    \begin{aligned}
        - \laplace u^{1}_{n+1} &= f,  &&\text{in } \fulldomain_1, \hspace*{2cm} 
        &- \laplace u^{2}_{n+1} &= f,  &&\text{in } \fulldomain_2, \\
        u^{1}_{n+1} &= 0 , &&\text{on } \dfulldomain \cap \dfulldomain_1,
        &u^{2}_{n+1} &= 0 , &&\text{on } \dfulldomain \cap \dfulldomain_2, \\
        u^{1}_{n+1} &=  u^{2}_{n} , &&\text{on } \Gamma_1,
        &u^{2}_{n+1} &=  u^{1}_{n} , &&\text{on } \Gamma_2 .
    \end{aligned}
\end{equation} 
This iterative scheme is known as the parallel Schwarz method and dates back to~\cite{Lions87}.
As described in \eqref{eq:SubproblemsParallelSchwarz}, the method is typically formulated at the continuous level, although it is ultimately combined with a spatial discretization.

The convergence analysis of this method, as well as of many related domain decomposition schemes, is often carried out either continuously in space, that is, prior to the application of a spatial discretization (see, e.g., \cite[Section~2]{GanZh22}), or after discretization, at the matrix level (see, e.g., \cite[Chapter~2]{TosWi05}).  
In contrast, \Cref{Thm:discreteSaintVenant} provides another  approach to proving the convergence of a finite element discretization of \eqref{eq:SubproblemsParallelSchwarz}, which incorporates the spatial discretization directly.

For simplicity, we consider a rectangular domain \(\fulldomain\) with two fixed overlapping subdomains $\fulldomain_1$ and $\fulldomain_2$, and denote by $\ell_{\mathrm{ov}}$ the number of mesh layers in the overlap; see \Cref{Fig:OvDecomp} for an illustration.
\begin{figure}[t]
    \centering
    \includegraphics{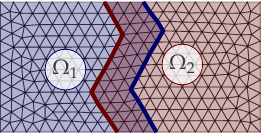}
    \hspace{0.5cm}
    \includegraphics{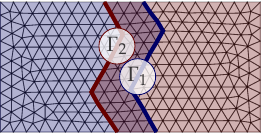}
    \caption{Left: Example decomposition of a rectangular domain into two overlapping subdomains $\fulldomain_1$ and $ \fulldomain_2$ with $\ell_{\mathrm{ov}}=3$. Right: Interfaces $\Gamma_1$ and $\Gamma_2$ of the subdomains.}
    \label{Fig:OvDecomp}
\end{figure}

The finite element discretization of \eqref{eq:SubproblemsParallelSchwarz} then iteratively seeks a sequence of solutions
\begin{subequations}\label{eq:SubproblemsDiscreteParallelSchwarz}
\begin{equation}
    \bigl(u_{h,n}^{1}, u_{h,n}^{2}\bigr) \in \PkTh[\fulldomain_1]\times\PkTh[\fulldomain_2],\qquad n\geq 0, 
\end{equation}
solving the two problems
\begin{equation}
    \begin{aligned}
        a_{\fulldomain_1}(u^{1}_{h,n+1}, \varphi_h^{1}) 
            &= \ip[\fulldomain_1]{f, \varphi_h^{1}},
            \quad\text{for all } \varphi_h^{1} \in \PkZeroTh[\fulldomain_1]\\
        \restr{u^{1}_{h,n+1}}{\Gamma_1} &= \restr{u_{h,n}^{2}}{\Gamma_1},\\
        a_{\fulldomain_2}(u^{2}_{h,n+1}, \varphi_h^{2}) 
            &= \ip[\fulldomain_2]{f, \varphi_h^{2}},
            \quad\text{for all } \varphi_h^{1}\in \PkZeroTh[\fulldomain_2]\\
        \restr{u^{2}_{h,n+1}}{\Gamma_2} &= \restr{u_{h,n}^{1}}{\Gamma_2},\\
    \end{aligned}
\end{equation}
\end{subequations}
with some initial values \(u_{h,0}^{1}\) and \(u_{h,0}^{2}\).
Note, that we do not rely on any specific lifting of the boundary datum.
The usual nodal lifting is sufficient and preferred for implementation.

We aim to prove that our iteration is a contraction when compared to a global discretization of \eqref{eq:globalProblemParallelSchwarz}.
Hence, we define \(u_h\in \PkZeroTh\) as the solution of
\begin{equation}\label{eq:globalDiscreteProblemParallelSchwarz}
    \abil[\fulldomain]{u_h, \varphi_h} 
        = \ip[\fulldomain]{f,\varphi_h},
    \qquad \text{for all }\varphi_h \in \PkZeroTh.
\end{equation}

The discrete error of the $n$-th iterates is now defined as
\begin{equation}\label{eq:discreteSchwarzEnergies}
    \mathcal{E}_{h,n} \coloneqq \sum_{i=1,2} \bignorm{u^{i}_{h,n} - u_h}_{a,\fulldomain_i}^2 = \sum_{i=1,2}\abs{u^{i}_{h,n} - u_h}^2_{H^1(\fulldomain_i)} \; .
\end{equation}

\begin{Theorem}\label{thm:SchwarzTheorem}
    If \(\ell_{\text{ov}}\geq 0\) is sufficiently large, then
    \begin{equation*}
        \mathcal{E}_{h,n+1} \leq  \theta \mathcal{E}_{h,n} \leq \theta^{n+1} \mathcal{E}_{h,0},
        \quad n\geq 0,
    \end{equation*} 
    where \(\theta < 1\).
\end{Theorem}
The contraction factor $\theta < 1$ is exponentially small in the number of mesh layers $\ell_{\mathrm{ov}}$ of the overlap. 
Thus, even a modest overlap typically guarantees rapid convergence, while larger overlaps further improve the contraction.
\begin{proof}[Proof of Theorem~\ref{thm:SchwarzTheorem}]
    We define the differences \(d_{h,n}^{i} = u_{h,n}^{i} - \restr{u_h}{\Omega^{i}}\), for \(i=1,2\), of the domain decomposition solution \eqref{eq:SubproblemsDiscreteParallelSchwarz} and the global finite element solution \eqref{eq:globalDiscreteProblemParallelSchwarz}.
    Combining the defintions of  \eqref{eq:SubproblemsDiscreteParallelSchwarz} and \eqref{eq:globalDiscreteProblemParallelSchwarz} shows that the differences solve the problems
    \begin{equation}\label{eq:DifferenceVaris}
    \begin{aligned}
        a_{\fulldomain_1}(d^{1}_{h,n+1}, \varphi_h^{1}) 
            &= 0,
            \quad\text{for all } \varphi_h^{1} \in \PkZeroTh[\fulldomain_1]\\
        \restr{d^{1}_{h,n+1}}{\Gamma_1} &= \restr{d_{h,n}^{2}}{\Gamma_1},\\
        a_{\fulldomain_2}(d^{2}_{h,n+1}, \varphi_h^{2}) 
            &= 0,
            \quad\text{for all } \varphi_h^{1}\in \PkZeroTh[\fulldomain_2]\\
        \restr{d^{2}_{h,n+1}}{\Gamma_2} &= \restr{d_{h,n}^{1}}{\Gamma_2},\\
    \end{aligned}
    \end{equation}
    for \(n\geq 0\).

    We note that any discrete lifting, which is a projection on the discrete trace space \(\restr{\PkTh[\fulldomain_i]}{\Gamma^i}\cap H^{1/2}_{00}(\Gamma_i)\), produces the same sequence of solutions, since the associated bilinearforms are coercive and hence, the solution is unique.
    In order to use a sharp a-priori bound of the discrete solution, we choose a localized lifting.
    We fix a localization parameter \(\epsilon > 0\) independent of \(h\) and choose liftings \(d_{h,\epsilon,n+1}^i \in \PkTh[\fulldomain_i]\cap H^1_{\partial\Omega_i\setminus\Gamma_i}(\Omega_i)\) such that
    \begin{equation*}
        \restr{d_{h,\epsilon,n+1}^1}{\Gamma_1} = \restr{d_{h,n}^2}{\Gamma_1},
        \qquad \restr{d_{h,\epsilon,n+1}^2}{\Gamma_2} = \restr{d_{h,n}^1}{\Gamma_2},
    \end{equation*}
    with 
    \begin{equation*}
        \supp d_{h,\epsilon,n+1} 
            \subset \Omega_{i,\epsilon} = \setc{x\in \Omega_i}{\operatorname{dist}(x,\Gamma_i) \leq \epsilon}.
    \end{equation*}
    A suitable lifting is constructed in Appendix~\ref{sec:LocalBoundaryLifting}, which is based on a Scott-Zhang interpolation operator together with a localization argument.
    In \Cref{cor:SupportedDiscreteLifting}, we prove the bounds 
    \begin{equation}\label{eq:LiftingTraceInequalities}
        \begin{aligned}
            \bignorm{d_{h,\epsilon,n+1}^1}_{H^1(\Omega_1)}
                &\leq C_1\epsilon^{-1/2} \bignorm{d_{h,n}^2}_{H^{1/2}_{00}(\Gamma_1)},\\
            \bignorm{d_{h,\epsilon,n+1}^2}_{H^1(\Omega_2)}
                &\leq C_2\epsilon^{-1/2} \bignorm{d_{h,n}^1}_{H^{1/2}_{00}(\Gamma_2)},
        \end{aligned}        
    \end{equation}
    with constants \(C_1,C_2>0\) that only depend on the shape regularity of the mesh and not on \(h\) and \(\epsilon\).

    Combining the definition of the discrete error \eqref{eq:discreteSchwarzEnergies}, the a-prior estimate \eqref{eq:APrioriEstimateDiscreteSolution}, and the inequalities \eqref{eq:LiftingTraceInequalities}, we obtain the estimate
    \begin{equation*}
        \begin{aligned}
        \mathcal{E}_{h,n+1} 
        &=\sum_{i=1,2}\abs{u^{i}_{h,n+1} - u_h}^2_{H^1(\fulldomain_i)} 
        = \sum_{i=1,2}\abs{d_{h,n+1}^i}^2_{H^1(\fulldomain_i)} \\
        &\leq 4 \Bigl( \norm{d_{h,\epsilon, n+1}^1}^2_{H^1(\fulldomain_1)} + \norm{d_{h,\epsilon, n+1}^2}^2_{H^1(\fulldomain_2)}  \Bigr) \\
        &\leq 4 \epsilon^{-1} \Bigl(C_1^2 \norm{d_{h,n}^2}^2_{H^{1/2}_{00}(\Gamma_1)} + C_2^2 \norm{d_{h,n}^1}^2_{H^{1/2}_{00}(\Gamma_2)} \Bigr)
        \end{aligned}
    \end{equation*}
    The trace inequality now allows swapping from one subdomain \(\fulldomain_j\) to its complement $\fulldomain \setminus \fulldomain_i$, that is
    \begin{equation*}
        \bignorm{d_{h,n}^j}^2_{H^{\half}_{00}(\Gamma_i)}
        =\bignorm{u^{j}_{h,n} - u_h}^2_{H^{\half}_{00}(\Gamma_i)} 
        \leq C_{\text{tr}} \abs{u^{j}_{h,n} - u_h}_{H^1(\fulldomain  \setminus \fulldomain_i)}^2, 
        \quad i=1,2, \text{ and } j \neq i.
    \end{equation*}
    See also \Cref{Fig:OvDecompSetminus} for an illustration.
    \begin{figure}[t]
        \centering
        \includegraphics{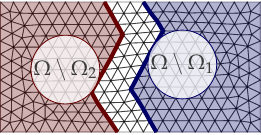}
        \caption{The complement domains $\fulldomain \setminus \fulldomain_i$ in the example decomposition of \Cref{Fig:OvDecomp}. Note, that $\fulldomain \setminus \fulldomain_1 \subset \fulldomain_2$ and vice versa.}
        \label{Fig:OvDecompSetminus}
    \end{figure}
    We now use the decay estimate of \cref{Thm:discreteSaintVenant}\ref{Thm:discreteSaintVenant:B} to enlarge the complementary domain to the full subdomain, gaining a contraction factor depending on the size of the overlap \(\ell_{\text{ov}}\). That is,
    \begin{equation*}
        \abs{u^{j}_{h,n} - u_h}_{H^1(\fulldomain  \setminus \fulldomain_i)}^2 \leq \rho^{\ell_{\text{ov}}} \abs{u^{j}_{h,n} - u_h}_{H^1(\fulldomain_j)}^2, 
    \end{equation*}
    with $\rho = M_0 / (1+ M_0) < 1$ for $j=1,2$ and $i \neq j$.
    Combining the previous estimates then yields
    \begin{align*}
        \mathcal{E}_{h,n+1} 
        \leq \! \frac{4}{\epsilon}\max\monosetc{C_1^2,C_2^2} C_{\text{tr}} \rho^{\ell_{\text{ov}}} 
        \Bigl( \abs{u^{1}_{h,n} - u_h}_{H^1(\fulldomain_1)}^2 \! + \abs{u^{2}_{h,n} - u_h}_{H^1(\fulldomain_2)}^2\Bigr) \!
        = C \rho^{\ell_{\text{ov}}}  \mathcal{E}_{h,n}. 
    \end{align*}
    Hence, if $\ell_{\mathrm{ov}}$ is sufficiently large such that 
    \begin{equation*}
        \theta \coloneqq 4\epsilon^{-1} \max\monosetc{C_1^2,C_2^2} C_{\text{tr}} \, \rho^{\ell_{\mathrm{ov}}} < 1,
    \end{equation*}
    the method converges independently of the initial iterates $u^{1}_{h,0}$ and $u^{2}_{h,0}$, which proves the claim.
\end{proof}

\bigskip

\section{Numerical validation}
\label{sec:NumericalValidation}
In this section, we numerically validate the discrete energy decay proven in \Cref{sec:MainResults}, namely
\begin{equation*}
    E_{\Dh}(\ell) \leq \rho E_{\Dh}(\ell+1), \quad \ell < \ellmax,
\end{equation*}
where $\rho<1$ denotes the contraction factor from \Cref{cor:decayfactor} and the discrete energy $E_{\Dh}(\ell)$ is defined in \eqref{eq:DefEnergyOnPatches}. 
Moreover, we compare the observed decay rates with the asymptotic scaling predicted by the estimate in \Cref{cor:asymp}.

For each numerical example, we consider a domain $\Omega$ equipped with a
conforming triangulation $\Th[\Omega]$ and fix the polynomial degree $k$ of the finite element discretization.
On $\Gamma \subseteq \partial\Omega$, we impose
inhomogeneous Dirichlet boundary data~$g$, while homogeneous Dirichlet
conditions are prescribed on the complementary part
$\Gamma^{C}$. Furthermore, we specify a starting
patch $\Dh \subset \Th = \Th[\Omega]$. Throughout this section we denote the discrete solution by
\[
u_h = u_{h0} + u_{hg} \in \PkTh[],
\]
where $u_{h0}$ denotes the solution of \eqref{eq:variationalFormulationDiscrete}.

\subsection*{\texorpdfstring{Measurement of the contraction factor $\rho$}{Measurement of the contraction factor}}
The total energy of the discrete solution~$u_h$ on $\fulldomain$ is given by 
 $\energynorm{u_h}_{\fulldomain}^2$, cf.~\eqref{eq:NaturalEnergyNorm}.
For each layer $\ell=0,\dots,\ell_{\max}$, we construct the corresponding patch $\mathcal{P}_\ell(\Dh)$ and compute the localized energy $E_{\Dh}(\ell)$ according to \eqref{eq:DefEnergyOnPatches}. 
For comparability, we introduce the relative energy
\begin{equation}
    \label{def:relativeEnergy}
    E_{\Dh}^{\mathrm{rel}}(\ell) \coloneq \frac{E_{\Dh}(\ell)}{\energynorm{u_h}_{\fulldomain}^2}, \quad \text{for } 0 \leq \ell \leq \ellmax
.\end{equation}
A semilogarithmic plot of the values
$E_{\Dh}^{\mathrm{rel}}(\ell)$ already provides a first indication of the
expected exponential decay.

To determine the contraction factor quantitatively, we form the quotient
\begin{equation*}
    Q_{\Dh}(\ell) = \frac{E_{\Dh}^{\mathrm{rel}}(\ell)}{E_{\Dh}^{\mathrm{rel}}(\ell+1)}, \quad \text{for } 0 \leq \ell < \ellmax.
\end{equation*}
For a fixed layer $\ell$, this quotient measures the ratio of the energy
contained in the patch $\Patch{\ell}{\Dh}$ to that of the next larger patch.
Equivalently, it represents the local slope of the relative energy increase from
layer $\ell$ to layer $\ell+1$.

In practice, the values of $Q_{\Dh}(\ell)$ may exhibit mild fluctuations
for the first and the last few layers. In our experiments, we observed
slightly smaller quotients for small $\ell$, while the quotients tend to
increase for large $\ell$ as the patches approach the full domain
$\fulldomain$.
To obtain a single representative value of the contraction factor for a
given configuration, we therefore define $\widehat{\rho}$ as the median of
the set $\{Q_{\Dh}(\ell)\}_{0 \leq \ell < \ellmax}$, i.e.,
\begin{equation}
    \label{def:rhohat}
    \widehat{\rho} 
    \coloneqq \mathrm{median} \setc{Q_{\Dh}(\ell)}{0 \leq \ell < \ellmax} .
\end{equation}
This quantity provides a robust estimate of the slope of the normalized
energies for the majority of layers as verified later.

\subsection{Rectangular domain with inhomogeneous boundary conditions on the left}

As a first numerical example, we consider the rectangular domain 
\begin{equation*}
    \fulldomain = [0,2] \times [0,1],
\end{equation*}
with boundary decomposition $\partial\fulldomain = \Gamma \cup \Gamma^{C}$,
where
\begin{equation*}
    \Gamma = \partial\fulldomain \cap \{ (x,y) \in \mathbb{R}^2 : x = 0 \}.
\end{equation*}
The domain is discretized by a non-equidistant triangulation
$\Th = \Th[\fulldomain]$ generated using gmsh (cf. \cite{GeuRe09}), and we employ linear finite elements. The resulting mesh with target mesh size $\htar = 0.1$ is shown in \Cref{fig:rectangle}, with $h_{\min} \approx 0.089$ and $h = h_{\max} \approx 0.117$. For the later experiments, finer meshes were generated using smaller values of $\htar$.

\begin{figure}[t]
    \begin{center}
        \includegraphics{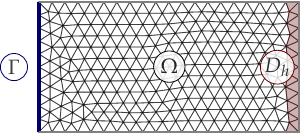}
    \end{center}
    \caption{Triangular finite element mesh of the rectangular domain $\fulldomain$ with inhomogeneous boundary conditions prescribed on $\Gamma$ (left boundary). The initial cell set $\Dh$ is highlighted in red. The displayed mesh corresponds to target mesh size $\htar=0.1$.
    Simulations were performed on independently generated finer meshes.}
    \label{fig:rectangle}
\end{figure}

\begin{figure}[t]
    \begin{center}
        \includegraphics{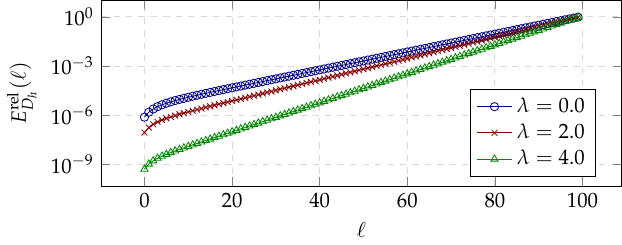}
    \end{center}
    \caption{Relative energy $E_{\Dh}^{\mathrm{rel}}(\ell)$ on the patches $\Patch{\ell}{\Dh}$ for $0 \le \ell \le \ell_{\max}$ shown in a semilogarithmic plot for different values of $\lambda$. The approximately linear behavior indicates exponential growth with respect to $\ell$. The underlying mesh was generated with target mesh size $h_{tar}=0.02$.}
    \label{fig:rectangle-NormedEnergies}
\end{figure}

\begin{figure}[ht]
    \begin{center}
        \includegraphics{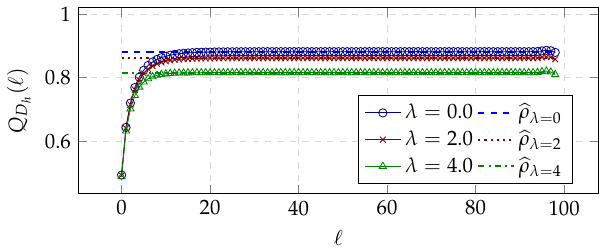}
    \end{center}
    \caption{Layer-wise quotients of relative energies $Q_{\Dh}(\ell)$ for $0 \le \ell < \ell_{\max}$ and different values of $\lambda$ (solid lines). The horizontal dashed, dotted, and dash-dotted lines indicate the corresponding median values $\widehat{\rho}$. The underlying mesh was generated with target mesh size $\htar = 0.02$.}
    \label{fig:rectangle-Quotients}
\end{figure}

On $\Gamma \subseteq \partial \fulldomain$ we prescribe the inhomogeneous
Dirichlet datum
\[
g(y) = \sin(\pi y).
\]
We first verify that the proposed procedure for measuring the contraction
factor $\widehat{\rho}$ is justified; see
\Cref{fig:rectangle-NormedEnergies,fig:rectangle-Quotients}, where we depict the relative energies $E_{\Dh}^{\mathrm{rel}}(\ell)$ and the per-layer quotients $Q_{\Dh}(\ell)$, respectively.
Moreover, the data exhibit extremely low dispersion around the median,
which is quantified by the median absolute deviation (MAD), which is defined as 
\begin{equation*}
   \mathrm{MAD} \coloneq \mathrm{median}\setc{\abs{Q_{\Dh}(\ell)-\widehat{\rho}}}{0 \le \ell < \ell_{\max}} 
\end{equation*}
and reported in \Cref{Table:MAD}. 

\setlength{\tabcolsep}{12pt}
\begin{table}[ht]
    \centering
    \begin{tabular}[t]{l|rr}
        \hline
        &\multicolumn{1}{c}{$\widehat{\rho}$} 
        & \multicolumn{1}{c}{MAD}\\
        \hline
        $\lambda = 0.0 $ &  $0.882$ & $4.44\cdot10^{-05}$ \\
        $\lambda = 2.0 $ &  $0.862$ & $4.33\cdot10^{-05}$ \\
        $\lambda = 4.0 $ &  $0.816$ & $2.65\cdot10^{-05}$ \\
        \hline
    \end{tabular}
    \vspace{4pt}
    \caption{Median contraction factor $\widehat{\rho}$ and MAD of the layer-wise quotients $Q_{\Dh}(\ell)$ for different values of $\lambda$. The MAD quantifies the typical deviation of the quotients from their median.}

    \label{Table:MAD}
\end{table}

The small MAD values (of order $10^{-5}$) indicate that the layer-wise 
quotients $Q_{\Dh}(\ell)$ exhibit only negligible fluctuations around their 
median, apart from the first few layers. 
This quantitatively confirms the strong alignment observed in 
\Cref{fig:rectangle-Quotients} and supports the interpretation of $\widehat{\rho}$ as a robust representative contraction factor.

\begin{figure}[ht]
    \centering
    \begin{minipage}[t]{0.55\textwidth}
        \centering
        \includegraphics{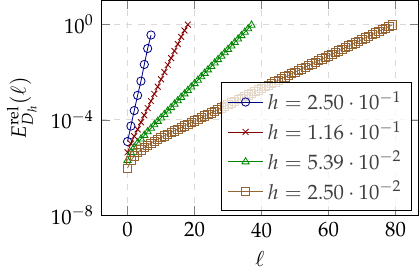}
    \end{minipage}
    \begin{minipage}[t]{0.44\textwidth}
        \centering
        \includegraphics{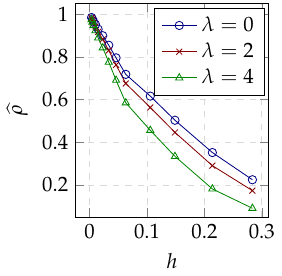}
    \end{minipage}
    \caption{Left: Relative energies $E^{\mathrm{rel}}_{\Dh}(\ell)$ in a semilogarithmic plot versus the layer index $\ell$ for $\lambda = 0$ and different mesh sizes $h$. 
    Right: Measured contraction factor $\widehat{\rho}$ plotted against $h$ for different values of $\lambda$.} 
    \label{fig:rectangle-NormedEnergies-h}
\end{figure}

\begin{figure}[ht]
    \begin{center}
        \includegraphics{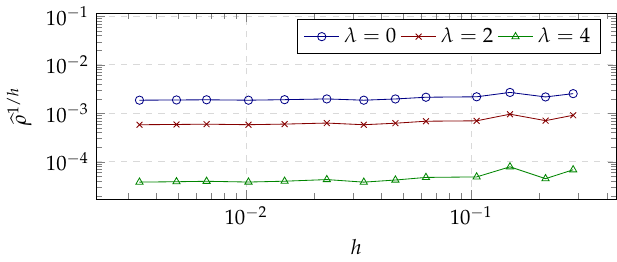}
    \end{center}
   \caption{Scaled contraction factor $\widehat{\rho}^{\,1/h}$ plotted against the mesh size $h$ in a double logarithmic scale for different values of $\lambda$. 
    As $h \to 0$, the scaled contraction factor approaches a constant, in agreement with the asymptotic prediction of \Cref{cor:asymp}.}
    \label{fig:rectangle-SlopeScaled-h}
\end{figure}

Next, our goal is to investigate the dependence of the measured contraction factor $\widehat{\rho}$ on the
mesh size $h$. To this end, we perform a sequence of computations on
successively refined meshes; see
\Cref{fig:rectangle-NormedEnergies-h}.

The numerical results indicate that the decay factor per layer approaches
$1$ as $h \to 0$. At first glance, this behavior appears to contradict the
discrete decay result. However, this effect can be explained by the fact
that both the diameter of the starting patch $\Dh$ and the physical width
of each layer $\bLayer{\ell}$ decrease with $h$. At the same time, the
maximal layer index $\ell_{\max}$ increases as the mesh is refined, so that
the discrete Saint-Venant principle from \Cref{Thm:discreteSaintVenant} can
be applied over a larger number of layers.

Motivated by the asymptotic estimate in \Cref{cor:asymp}, we therefore
rescale the measured decay by raising the data to the power $1/h$. The
resulting behavior is shown in \Cref{fig:rectangle-SlopeScaled-h} and we see that the scaled rate indeed behaves like a constant, when $h\to 0$.

\subsection{Hexagonal domain with inhomogeneous boundary conditions on all sides}

We consider a hexagonal domain centered at
the origin, where the distance from each vertex to the origin is equal to
one. 
On the entire boundary $\Gamma = \partial\fulldomain$, we prescribe the
inhomogeneous Dirichlet condition $g \equiv 1$. 
The mesh is constructed such that it contains a
structured submesh covering the square $[-0.2,0.2] \times [-0.2,0.2]$, see \Cref{fig:hexagon}. The starting patch $\Dh$ is
chosen as the collection of all elements $K \in \Th$ belonging to this
embedded submesh.
By this construction, the geometric support
$\bigcup_{K \in \Dh} \overline{K}$ remains of fixed physical size under mesh
refinement.

\begin{figure}[ht]
    \begin{center}
        \includegraphics{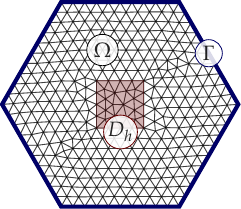}
    \end{center}
    \caption{Triangular finite element mesh of the hexagonal domain $\fulldomain$ with inhomogeneous Dirichlet boundary condition $g \equiv 1$ prescribed on $\Gamma = \partial\fulldomain$. 
    The starting patch $\Dh$ (red) consists of the elements of the structured submesh covering the square $[-0.2,0.2]^2$. 
    The depicted mesh corresponds to target mesh size $\htar = 0.1$.}
    \label{fig:hexagon}
\end{figure}

\begin{figure}[ht]
    \begin{center}
        \includegraphics{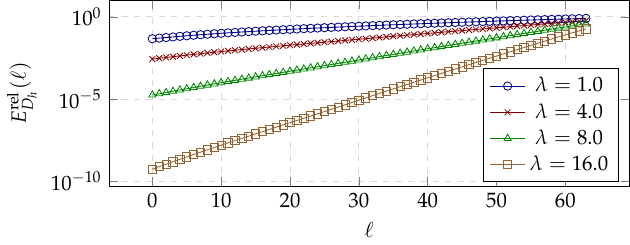}
    \end{center}
    \caption{Relative energies $E^{\mathrm{rel}}_{\Dh}(\ell)$ plotted against the layer index $\ell$ in a semilogarithmic scale for different values of $\lambda$. 
    The starting patch $\Dh$ has fixed physical support. 
    The underlying mesh corresponds to target mesh size $\htar = 0.01$.}
    \label{fig:hexagon-energies}
\end{figure}

\begin{figure}[ht!]
    \begin{center}
        \includegraphics{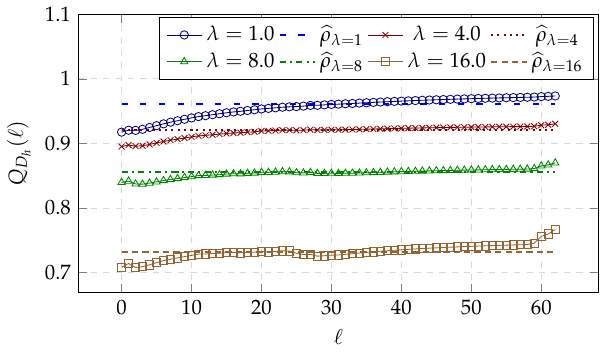}
    \end{center}
    \caption{Layer-wise quotients $Q_{\Dh}(\ell)$ for $0 \le \ell < \ell_{\max}$ and different values of $\lambda$ (solid lines). 
    The horizontal dashed lines indicate the corresponding median contraction factors $\widehat{\rho}_{\lambda}$. 
    The starting patch $\Dh$ has fixed physical support. 
    Mesh target size $\htar = 0.01$.}
    \label{fig:hexagon-quotients}
\end{figure}

\begin{figure}[ht]
    \begin{center}
        \includegraphics{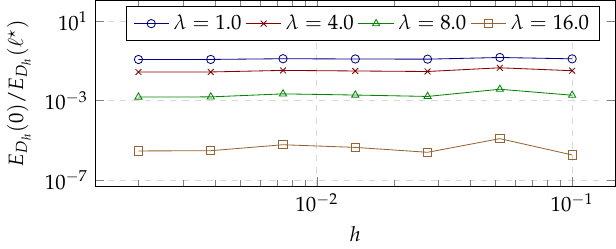}
    \end{center}
    \caption{Energy ratio $E_{\Dh}(0) / E_{\Dh}(\ell^{\star})$ plotted against the mesh size $h$ for fixed physical width $\delta = \ell^{\star} h_{\min} = 0.3$ and different values of $\lambda$. 
    The near constant behavior as $h \to 0$ confirms the exponential decay with respect to physical distance as stated in  \Cref{cor:asymp}.}

    \label{fig:hexagon-energydropfixedpatch}
\end{figure}

\begin{figure}[ht]
    \begin{center}
        \includegraphics{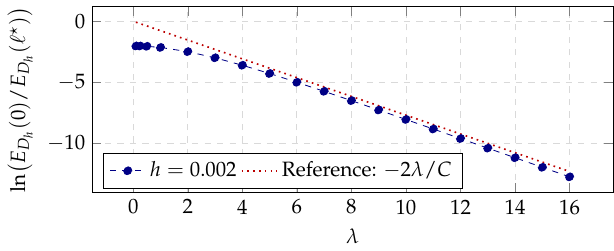}
    \end{center}
    \caption{Dependence of $\ln\!\bigl(E_{\Dh}(0)/E_{\Dh}(\ell^{\star})\bigr)$ on $\lambda$ for $\htar = 0.002$. 
    The dotted line shows the reference slope $-2\lambda/C$ with $C = 2.6$, in accordance with \Cref{cor:asymp}. 
    The linear behavior confirms the predicted exponential dependence of the decay rate on $\lambda$.}
    \label{fig:hexagon-energydropfixedpatchlambda}
\end{figure}

In \Cref{fig:hexagon-energies}, we plot the relative energies
$E^{\mathrm{rel}}_{\Dh}(\ell)$ for a fixed mesh size $h$ and all layers
$0 \leq \ell < \ellmax$, considering several values of the parameter
$\lambda$. The approximately linear behavior confirms exponential decay per layer. The corresponding quotients $Q_{\Dh}(\ell)$ are shown in
\Cref{fig:hexagon-quotients}, where we additionally indicate the median
value $\widehat{\rho}$ for each choice of $\lambda$.

Overall, the measurement $\widehat{\rho}$ still provides a reasonable
characterization of the decay behavior, in particular for larger values of
$\lambda$. For small $\lambda$, however, the median quotient may be
pessimistic in the sense that the decay in the first few layers is even
stronger than suggested by $\widehat{\rho}$; see
\Cref{fig:hexagon-energies,fig:hexagon-quotients}.

Rather than rescaling the median quotient by a power of $1/h$, we therefore
adopt the more geometric point of view discussed in \Cref{Rem:Interpretation-rho-scaling}. Motivated by the fixed physical size
of the initial patch $\Dh = \Patch{0}{\Dh}$, we introduce a layer index
$\ell^{\star} = \ell^{\star}(h)$ such that the enlarged patch
$\Patch{\ell^{\star}}{\Dh}$ covers a neighborhood of $\Dh$ with fixed
physical width. More precisely, we fix a parameter
$\delta \coloneqq \ell^{\star} h_{\min}$, where $h_{\min}$ denotes the
minimal element diameter, so that $\delta$ represents the additional
physical width of the patch $\Patch{\ell^{\star}}{\Dh}$ compared to the
initial cell set $\Dh$.  

In analogy with the asymptotic estimate in \Cref{cor:asymp}, we therefore expect the ratio of energies on
$\Dh$ and $\Patch{\ell^{\star}}{\Dh}$ to remain bounded as $h \to 0$.

To verify this, we plot the energy ratio
$E_{\Dh}(0) / E_{\Dh}(\ell^{\star})$ against $h$ for a fixed value
$\delta \equiv 0.3$; see \Cref{fig:hexagon-energydropfixedpatch}. Finally,
\Cref{fig:hexagon-energydropfixedpatchlambda} shows that the logarithm of
this ratio depends linearly (with negative slope) on $\lambda$, in full
agreement with the predicted structure of the constant in
\Cref{cor:asymp}.

\subsection*{Conclusion}
The numerical experiments fully confirm the discrete Saint-Venant principle and its asymptotic interpretation. 
The relative energies exhibit a clear exponential decay with respect to the layer index, and the layer-wise quotients remain remarkably stable, as quantified by the small median absolute deviations. 
Under mesh refinement, the measured contraction factor behaves consistently with the predicted scaling, approaching unity per layer while yielding a constant decay rate per physical distance. 
In the fixed-width setting, the energy drop across a prescribed physical neighborhood remains bounded as $h \to 0$, and its logarithm depends linearly on $\lambda$, in precise agreement with the structure predicted by \Cref{cor:asymp}. 
Together, these results provide comprehensive numerical validation of the theoretical decay estimates.

The code accompanying the numerical experiments is publicly available at
\begin{equation*}
    \text{\url{https://gitlab.kit.edu/tim.buchholz/dsv-numerical-validation.git}}\; .
\end{equation*}

\appendix

\section{Local boundary lifting}
\label{sec:LocalBoundaryLifting}

We are in the setting of \cref{subsec:ProblemAndVariational}, i.e., \(\fulldomain\subset \R^d\) is a polyhedral domain and \(\Gamma\subset \partial\Omega\) a simply connected subset of the boundary with complement \(\Gamma^\text{C} = \partial\Omega \setminus \Gamma\). 
In this section, we assume that \(\Gamma \neq \emptyset\) and restrict ourselves to \(d=2\) for the sake of presentation.

The aim of this section is to construct a discrete lifting for \(g\in H^{1/2}_{00}(\Gamma)\), that is \(H^1\)-stable and localized in an \(\epsilon\)-neighborhood of \(\Gamma\).
We use a common construction by interpolating a harmonic extension of the trace on the strip.
Since we demand that the lifting is a projection with respect to \(\restr{\PkTh[]}{\Gamma}\), we use the Scott-Zhang quasi-interpolation \cite{ScoZha90}, which is precisely constructed for that purpose.
Furthermore, we explicitly track the thickness $\epsilon$ of the strip.

For a given geometry dependent \(\epsilon_0> 0\), we define for all \(\epsilon\in (0,\epsilon_0)\) the tubular subdomain
\begin{equation}\label{eq:TubularSubdomain}
    \Omega_\epsilon 
        = \setc{x \in \Omega}{ \operatorname{dist}(x,\Gamma)\leq \epsilon}
        \subsetneq \Omega.
\end{equation}
Note that a small enough \(\epsilon_0\) ensures that \(\partial \Omega_\epsilon \setminus \Gamma\) is an offset of straight-line segments connected with circular arcs of radius \(\epsilon\) for every vertex of \(\Gamma\).
Hence, \(\Omega_\epsilon\) is a Lipschitz domain.  

\begin{Lemma}\label{lem:ThinDomainElliptic}
    Let \(g\in H^{1/2}_{00}(\Gamma)\) and \(u_\epsilon \in H^1(\Omega_\epsilon)\) denote the harmonic extension, i.e.,
    \begin{equation}\label{eq:TubularPDE}
        -\laplace u_\epsilon =  0,
        \qquad \restr{u_\epsilon}{\Gamma} = g,
        \qquad \restr{u_\epsilon}{\partial\Omega_\epsilon\setminus\Gamma} = 0.
    \end{equation}
    There exists a constant \(C>0\) independent of \(\epsilon\) and \(g\) such that
    \begin{equation}\label{eq:TubularPDE:APrioriEstimate}
        \norm{u_\epsilon}_{H^1(\Omega_\epsilon)}
            \leq C\epsilon^{-1/2} \norm{g}_{H^{1/2}_{00}(\Gamma)}.
    \end{equation}
\end{Lemma}
The Lemma is proved at the end of this section.
We motivate the scaling \(\epsilon^{-1/2}\) in dimension one.
Let \(\Omega = [0,1]\) and \(\Gamma = \monosetc{0}\).
Then, the solution of \eqref{eq:TubularPDE} for \(g\in \R\) is given by \(u(x) = g(1 - \epsilon^{-1}x)\).
Hence, we obtain the scaling \(\norm{u'}_{L^2(\Omega_\epsilon)} = \epsilon^{-1/2} |g|\), which is sharp in one dimension.

With \Cref{lem:ThinDomainElliptic} and the Scott-Zhang interpolation operator in \cite{ScoZha90}, we can proof the existence of the desired discrete lifting.
\begin{Corollary}\label{cor:SupportedDiscreteLifting}
    Let \(2h \leq \epsilon\).
    For \(g\in H^{1/2}_{00}(\Gamma)\) there exists a discrete lifting \(u_{hg}\in \PkTh[]\cap H^1_{\Gamma^\text{C}}(\Omega)\) such that
    \begin{equation*}
        \norm{u_{hg}}_{H^1(\Omega)}
            \leq C \epsilon^{-1/2}\norm{g}_{H^{1/2}_{00}(\Gamma)},
        \qquad \supp u_{hg} \subset \Omega_{\epsilon},
    \end{equation*} 
    where the constant \(C>0\) is independent of \(\epsilon\), \(g\) and \(h\), but depends on the shape-regularity.
    For \(g\in \restr{\PkTh[]}{\Gamma}\cap H^{1/2}_{00}(\Gamma)\) the lifting is a projection, i.e., \(\restr{u_{hg}}{\Gamma} = g\). 
\end{Corollary}
Note that \(2h\leq \epsilon\) is not a restriction, since it only prohibits degenerate cases where the mesh size is too large to make sense of the claim \(\supp u_{hg}\subset \Omega_\epsilon\).
\begin{proof}[Proof of \Cref{cor:SupportedDiscreteLifting}]
    Let \(\epsilon' = \epsilon/2\) and define \(u_{\epsilon'} \in H^1(\Omega_{\epsilon'})\) according to Lem\-ma~\ref{lem:ThinDomainElliptic}.
    Since it holds \(\restr{u_{\epsilon'}}{\partial\Omega_{\epsilon'} \setminus \Gamma} = 0\), we can extend by zero to the whole of the domain  $\Omega$.
    We denote this extension with \(\widetilde{u}_{\epsilon'} \in H^1_{\Gamma}(\Omega)\), satisfying the estimate
    \begin{equation*}
        \norm{\widetilde{u}_{\epsilon'}}_{H^1(\Omega)}
            = \norm{u_{\epsilon'}}_{H^1(\Omega_{\epsilon'})}
            \leq C\sqrt{2}\epsilon^{-1/2}\norm{g}_{H^{1/2}_{00}(\Gamma)},
        \qquad \supp \widetilde{u}_{\epsilon'} \subset \Omega_{\epsilon'}.
    \end{equation*} 

    Furthermore, let \(\Pi_h^{\text{SZ}} : H^1(\Omega) \to \PkTh[]\) denote the Scott-Zhang interpolation operator, defined in \cite[Sec.~2]{ScoZha90}.
    Then, the discrete lifting defined as \(u_{hg} = \Pi_h^{\text{SZ}} \widetilde{u}_{\epsilon'}\) fulfills the requirements.

    With the stability in \cite[Thm.~3.1]{ScoZha90} and \Cref{lem:ThinDomainElliptic} it holds that
    \begin{equation*}
        \norm{u_{hg}}_{H^1(\Omega)} 
            = \norm{\Pi_h^{\text{SZ}}\widetilde{u}_{\epsilon'}}_{H^1(\Omega)}
            \leq C\norm{\widetilde{u}_{\epsilon'}}_{H^1(\Omega)}
            \leq C\sqrt{2}\epsilon^{-1/2}\norm{g}_{H^{1/2}_{00}(\Gamma)}.
    \end{equation*}
    The discrete lifting is locally supported on \(\Omega_{\epsilon}\) since for any \(K\in \Th\) the value of the interpolation \(\restr{u_{hg}}{K}\) is completely controlled by its surrounding patch.
    Additionally, it fulfills the desired properties on the boundary, as mentioned in \cite[Sec.~5]{ScoZha90}.
\end{proof}

It remains to prove \Cref{lem:ThinDomainElliptic}.
\begin{proof}[Proof of \Cref{lem:ThinDomainElliptic}]
Note that the extension by zero \(\widetilde{g} \in L^2(\partial \Omega_\epsilon)\) is an element of \(H^{1/2}(\partial\Omega_\epsilon)\) since \(g\in H^{1/2}_{00}(\Gamma)\) by assumption.
Hence, since \(\Omega_\epsilon\) is Lipschitz, the Lax-Milgram Lemma ensures that a unique weak solution \(u_\epsilon\in H^1(\Omega_\epsilon)\) exists.
Furthermore, by the Dirichlet principle, \(u_\epsilon\) is the minimizer of the functional
\begin{equation*}
    J\,\defpnt\, V_{\widetilde{g}} \to \R,\quad
        \phi \mapsto \frac{1}{2}\int_{\Omega_\epsilon} |\nabla \phi(x)|^2 \dintx,
    \qquad
        V_{\widetilde{g}} =\setc{\phi \in H^1(\Omega_\epsilon)}{\restr{\phi}{\partial\Omega_\epsilon} = \widetilde{g}}.
\end{equation*}
Thus, by the minimizing property and Poincaré inequality, it is sufficient to construct 
\begin{equation*}
    v\in V_{\widetilde{g}},\qquad\text{s.t.  } \norm{v}_{H^1(\Omega_\epsilon)} \leq C \epsilon^{-1/2} \norm{g}_{H^{1/2}_{00}(\Gamma)}
\end{equation*}
in order to show \eqref{eq:TubularPDE:APrioriEstimate}.
In the following, we construct such a suitable \(v\) with respect to a finite open cover of \(\Gamma\), following essentially ideas from \cite{Gri11}.

Let \(V_\Gamma = \monosetc{\alpha_1,\ldots \alpha_N}\subset \R^2\) denote the vertices of \(\Gamma\).
For every vertex we define
\begin{equation*}
    U_j = B_{\epsilon}(\alpha_j) = \setc{x\in\R^2}{|x-\alpha_j|< \epsilon},
    \qquad j=1,\ldots,N.
\end{equation*}
Note that we can always choose $\epsilon_0>0$ in \eqref{eq:TubularSubdomain} sufficiently small such that for all $\epsilon \in (0, \epsilon_0)$ the set $U_j \cap \Omega$ contains no vertex $\alpha_k$ with $k \neq j$, and no edge except for the two edges adjacent to $\alpha_j$.

The \(N\) vertices are connected by \(N-1\) straight line elements, and can each be covered by a tubular neighborhood \(V_m\) of size \(\epsilon_m>0\), for \(m=1,\ldots,N-1\), such that \(V_m\cap \Omega\) does not contain any vertex and any other line element, but overlaps with the balls at the neighboring vertices.
Note, that we need \(\epsilon_m \leq \epsilon\) in general to accomplish this construction.
Nonetheless, the additional restrictions on \(\epsilon_m\) only depend on the angles of adjacent edges.
Hence, there exists a \(C> 0\) depending only on the smallest interior angle of the domain such that
\begin{equation}\label{eq:ComparableEpsilon}
    \max_{m=1,\ldots, N-1} \epsilon_m^{-1}
        \leq C \epsilon^{-1},
\end{equation}
see also Fig.~\ref{Fig:Tubuluar}.

\begin{figure} 
    \centering
    \includegraphics{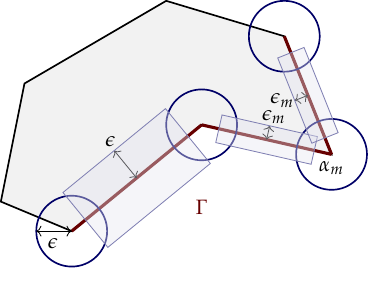}
    \caption{Finite cover of the boundary \(\Gamma\) for given \(\epsilon>0\). Note that \(\epsilon_m < \epsilon\) due to the small interior angle at the vertex \(\alpha_m\).}
    \label{Fig:Tubuluar}
\end{figure}

Thus, the two sets \(\monosetc{U_j}_{j=1}^N\) and \(\monosetc{V_j}_{j=1}^{N-1}\) are an open cover of \(\Gamma\), 
and we stress that \(N\) does not depend on \(\epsilon\) but only on the number of vertices.

We choose a partition of unity subordinate to the cover, i.e., \(\phi_j,\psi_m \in C^\infty_{\operatorname{c}}(\R^2)\) such that
\begin{equation}\label{eq:PartitionOfUnity}
    1 = \sum_{j=1}^{N} \phi_j(x) + \sum_{m=1}^{N-1} \psi_m(x),\,\,\,\, x\in \Gamma,
    \qquad \supp \phi_j \subset U_j,\quad \supp \psi_m \subset V_m.
\end{equation}
Next, we construct a local lifting in every corner.
Let \(j\in \monosetc{1,\ldots,N}\).
We define the cone \(K_j = \Omega\cap U_j\) of radius \(\epsilon\) at the corner \(\alpha_j\) and set \(g_j = \phi_j g\).
We denote the legs of \(K_j\) with \(\Gamma_j = \overline{K_j}\cap \Gamma\) and the arc with \(A_j\),
forming the boundary of \(K_j\).
Note, that the support of \(\phi_j\) is compact in \(U_j\) and, thus, the zero extension of \(g_j\) to \(\partial K_j\) is element of \(H^{1/2}(\partial K_j)\).
Hence, we conclude that there exists a unique \(v_j \in H^1(K_j)\) solving
\begin{equation}\label{eq:ConeProblem}
    -\laplace v_j = 0,\text{  in } K_j,
    \qquad \restr{v_j}{\Gamma_j} = g_j, \text{  on } \Gamma_j,
    \qquad \restr{v_j}{A_j} = 0, \text{  on } A_j.
\end{equation}
We prove the desired norm estimate with a scaling argument.
After possible rigid motion, we can assume that the cone has the polar-coordinate representation
\begin{equation*}
    K_j = \setc{(r,\theta)\in \R^+\times [0,2\pi]}{ 0 < r < \epsilon,\quad \theta \in (0,\varphi)},
\end{equation*}
for some angle \(\varphi\in (0,2\pi)\).

Let \(\widehat{K} = \setc{(r,\theta) \in \R^+\times [0,2\pi]}{0< r < 1,\quad \theta \in (0,\varphi)}\) denote the reference cone with angle $\varphi$,
and let \(T\,\defpnt\, \widehat{K} \to K_j\) be the reference map given in Cartesian coordinates by \((x, y) \mapsto (rx ,r y)\).
Then, the transformed function \(\widehat{v} = v_j \circ T\) solves
\begin{equation*}
    -\laplace \widehat{v} =0,\, \text{in } \widehat{K},
    \qquad \restr{\widehat{v}}{\widehat{\Gamma}} = \widehat{g},\,\text{on } \widehat{\Gamma},
    \qquad \restr{\widehat{v}}{\widehat{A}} = 0,\,\text{on } \widehat{A},
\end{equation*}
where \(\widehat{g} = g_j \circ T\), \(\widehat{\Gamma} = T^{-1}(\Gamma_j)\), and \(\widehat{A} = T^{-1}(A_j)\).
With the usual a priori estimate and the Poincaré inequality, we obtain the estimate
\(
    \norm{\nabla \widehat{v}}_{L^2(\widehat{K})} \lesssim \norm{\widehat{g}}_{H^{1/2}_{00}(\widehat{\Gamma})},
\)
where the hidden constant only depends on the angle \(\varphi\) but not on the radius \(\epsilon\).
From integral substitution, we further see that  \(\norm{\nabla v_j}_{L^2(K_j)} = \norm{\nabla \widehat{v}}_{L^2(\widehat{K})}\).
Furthermore, it holds that
\begin{equation}\label{eq:LTwoSchenkelScaling}
    \begin{aligned}
        \norm{\widehat{g}}_{L^2(\widehat{\Gamma})}^2
            &= \int_{0}^{1}
                |(g_j\circ T) (r,0)|^2 + |(g_j\circ T) (r,\varphi)|^2 \dint{r}\\
            &= \int_{0}^{1}
                |g_j (\epsilon r,0)|^2 + |g_j (\epsilon r,\varphi)|^2 \dint{r}\\
            &= \epsilon^{-1} \int_{0}^{\epsilon}
                |g_j(\widetilde{r},0)|^2 + |g_j (\widetilde{r},\varphi)|^2 \dint{\widetilde{r}}\\
            &= \epsilon^{-1}\norm{g_j}_{L^2(\Gamma_j)}^2.
    \end{aligned}
\end{equation}
For the semi-norm $\abs{\cdot}_{H^{1/2}(\widehat{\Gamma})}$, let \(\eta,\zeta \in \monosetc{0,\varphi}\).
With integral substitution, it holds that
\begin{equation}\label{eq:HEinHalbSubstitution}
    \begin{aligned}
        &\int_{0}^1 \int_{0}^1 \frac{\abs{(g_j\circ T)(r,\eta) - (g_j\circ T)(s,\zeta)}^2}{\abs{r - s}^2} \dint{r}\dint{s}\\
        &\quad = \int_{0}^1 \int_{0}^1 \frac{\abs{g_j(\epsilon r,\eta) - g_j(\epsilon s,\zeta)}^2}{\abs{r - s}^2} \dint{r}\dint{s}\\
        &\quad = \int_{0}^{\epsilon}\int_{0}^{\epsilon} \frac{\abs{g_j(\widetilde{r},\eta) - g_j(\widetilde{s},\zeta)}^2}{|\epsilon^{-1}\widetilde{r} - \epsilon^{-1}\widetilde{s}|^2} \epsilon^{-2} \dint{\widetilde{r}}\dint{\widetilde{s}}\\
        &\quad = \int_{0}^{\epsilon}\int_{0}^{\epsilon} \frac{\abs{g_j(\widetilde{r},\eta) - g_j(\widetilde{s},\zeta)}^2}{|\widetilde{r} - \widetilde{s}|^2} \dint{\widetilde{r}}\dint{\widetilde{s}}.
    \end{aligned}
\end{equation}
With the considerations in \cite[Sec.~1.5.2]{Gri11}, this shows that \(\abs{\widehat{g}}_{H^{1/2}(\widehat{\Gamma})} = \abs{ g_j}_{H^{1/2}(\Gamma_j)}\).
Repeating the substitution argument for the compatibility conditions, which we can express as weighted \(L^2\) integrals, shows that
\begin{equation}\label{eq:CompatibiltySubstitution}
\int_{0}^1 |1-r|^{-1} |(g_j\circ T)(r,\eta)|^2 \dint{r}
    = \int_{0}^\epsilon |\epsilon-\widetilde{r}|^{-1} |g_j(\widetilde{r},\eta)|^2\dint{\widetilde{r}}.\\
\end{equation} 
Thus, \eqref{eq:LTwoSchenkelScaling}--\eqref{eq:CompatibiltySubstitution} show that \(\norm{\widehat{g}}_{H^{1/2}_{00}(\widehat{\Gamma})} \leq \epsilon^{-1/2} \norm{g_j}_{H^{1/2}_{00}(\Gamma_j)}\).
Hence, we conclude that
\begin{equation}\label{eq:ScalingEstimateCone}
    \norm{\grad  v_j}_{L^2(K_j)} = \norm{\grad \widehat{v}}_{L^2(\widehat{K})}
        \lesssim \norm{\widehat{g}}_{H^{1/2}_{00}(\widehat{\Gamma})}
        \lesssim \epsilon^{-1/2} \norm{g_j}_{H^{1/2}_{00}(\Gamma_j)},
\end{equation}
with the hidden constant only depending on the angle \(\varphi\).

Next, we construct a lifting for every straight edge.
Let \(T_m = \Omega\cap V_m\) for \(m\in \monosetc{0,\ldots,N-1}\).
After possible rigid motion, we can assume that
\begin{equation*}
    T_m = (\delta^-_m, \delta^+_m) \times (0,\epsilon_m),
\end{equation*}
for \(\delta^\pm_m > 0\).
We define \(\Gamma_m = (\delta^-_m, \delta^+_m) \times \monosetc{0}\) and \(g_m = \psi_m g\).
Note that \(g_m \in H^{1/2}_{00}(\Gamma_m)\) since \(\psi_m\) is compactly supported.
Hence, there exists a unique \(w_m \in H^1(T_m)\) solving
\begin{equation*}
    -\laplace w_m = 0, \text{ in } T_m,
     \quad \restr{w_m}{\Gamma_m} = g_m, \text{ on } \Gamma_m,
     \quad \restr{w_m}{\partial T_m \setminus \Gamma_m } = 0, \text{ on } \partial T_m \setminus \Gamma_m.
\end{equation*}

In order to prove an estimate like above on the cone, we use an explicit representation of the solution.
Since \(g_m \in H^{1/2}_{00}(\Gamma_m)\), we find a sine series expansion of the data
\begin{equation*}
    g_m(x) = \sum_{n=1}^\infty g_m^n \sin\Bigl(\frac{\pi n}{L} (x-\delta^-)\Bigr),
\end{equation*}
where \(L = (\delta^+ - \delta^-)\) and \(g_m^n = \frac{2}{L} \int_{\delta^-}^{\delta^+} g_m(\zeta)  \sin\Bigl(\frac{\pi n}{L} (\zeta-\delta^-)\Bigr) \dint{\zeta}\), for \(n\geq 1\).
A separation of variables ansatz and matching this expansion at $\Gamma_{m}$ $(y=0)$ shows that the solution is of the form
\begin{equation*}
    w_m(x,y) = 
        \sum_{n=1}^{\infty} 
            g_m^n 
            \frac{\sinh\bigl(k_n(\epsilon_m - y)\bigr)}{\sinh(k_n \epsilon_m)}
            \sin\bigl(k_n(x-\delta^-)\bigr),
    \qquad k_n = \frac{n\pi}{L}.
\end{equation*}
Taking derivates and using standard orthogonality relations of trigonometric functions show that
\begin{equation*}
    \norm{\nabla w_m}_{L^2(T_m)}^2
        = \frac{L}{2} \sum_{n=1}^{\infty} |g_m^n|^2 k_n^2 \int_{0}^{\epsilon_m} \frac{\sinh^2\bigl(k_n y\bigr) + \cosh^2\bigl(k_n y\bigr)}{\sinh^2(k_n\epsilon_m)} \dint{y}.
\end{equation*}
To evaluate the latter integral, we use the identity \(\sinh^2(k_n y) + \cosh^2(k_n y) = \cosh(2 k_n y)\) and obtain
\begin{equation*}
    \frac{1}{\sinh^2(k_n \epsilon_m)} \int_{0}^{\epsilon_m} \cosh(2 k_n y) \dint{y}
        = \frac{1}{2 k_n}\frac{\sinh(2k_n \epsilon_m)}{\sinh^2(k_n \epsilon_m)}
        = \frac{1}{2 k_n} \coth(k_n \epsilon_m),
\end{equation*}
where the half-angle formula \(\sinh(2 k_n \epsilon_m) = 2 \sinh(k_n \epsilon_m)\cosh(k_n \epsilon_m)\) is used in the last equality.
Hence, we obtain the representation
\begin{equation*}
    \norm{\nabla w_m}_{L^2(T_m)}^2
        = \frac{L}{4} \sum_{n=1}^{\infty} |g_m^n|^2 k_n \coth(k_n \epsilon_m).
\end{equation*}
Application of the estimate \(\coth(k_n \epsilon_m) \leq 1 + (k_n \epsilon_m)^{-1}\) finally shows that
\begin{equation}\label{eq:ScalingEstimateStraightEdge}
    \norm{\nabla w_m}_{L^2(T_m)} \leq \epsilon^{-1/2}_m \norm{g_m}_{H^{1/2}_{00}(\Gamma_m)}.
\end{equation}

Putting everything together, we define \(v \in H^1(\Omega)\) by
\(
    v = \sum_{j=1}^{N} v_j + \sum_{m = 1}^{N-1} w_m.
\)
Note that $v$ is well-defined since any \(v_j\) and any \(w_m\) can be extended by zero to the whole of the domain, while remaining in \(H^1(\Omega)\) due to the homogenous Dirichlet boundary conditions at the interior boundary.
Furthermore, by the partition of unity \eqref{eq:PartitionOfUnity}, we conclude that
\begin{equation*}
    \restr{v}{\Gamma} 
        = \sum_{j=1}^N \restr{v_j}{\Gamma} + \sum_{m =1}^{N-1} \restr{w_m}{\Gamma}
        = \sum_{j=1}^N \phi_j g + \sum_{m=1}^{N-1} \psi_m g
        = g.
\end{equation*}
Finally, by \eqref{eq:ScalingEstimateCone}, \eqref{eq:ScalingEstimateStraightEdge}, \eqref{eq:ComparableEpsilon}, and the fact that the partition is finite, we obtain the desired estimate
\begin{equation*}
    \norm{\nabla v}_{L^2(\Omega)} \leq C \epsilon^{-1/2} \norm{g}_{H^{1/2}_{00}(\Gamma)},
\end{equation*}
with a constant \(C>0\), which is independent of \(\epsilon\).
This proves the claim.
\end{proof}

\section*{Acknowledgments}  
Funded by the Deutsche Forschungsgemeinschaft (DFG, German Research Foundation) -- Project-ID 258734477 -- SFB 1173.
The authors thank Benjamin~D\"orich and Roland~Maier for many helpful discussions and careful reading of the manuscript.
JD thanks the French Institute for Research in Computer Science and Automation (Inria) in Lille for the kind hospitality during his research stay. This research stay was financially supported by the Karlsruhe House of Young Scientists (KHYS).

\bibliographystyle{amsplain}	
\bibliography{bibliography}
	
\end{document}